\documentclass[11pt]{article}
\usepackage[paper=a4paper]{geometry}
\usepackage{amsmath, amsthm, amssymb}
\usepackage{multicol}
\usepackage{amsrefs}
\usepackage[colorlinks=true, linktoc=page, linkcolor=webgreen, citecolor=webgreen, urlcolor=webbrown]{hyperref}
\usepackage{xcolor}
\usepackage[mathscr]{eucal}
\usepackage{enumitem}
\usepackage{thmtools}
\usepackage{geometry} 
\usepackage{graphicx} 
\usepackage{float} 
\usepackage{wrapfig}
\usepackage{tikz-cd}
\usepackage{ytableau}
\ytableausetup{centertableaux}
\usepackage[utf8]{inputenc} 
\usepackage[T1]{fontenc}

\newtheorem{thm}{Theorem}[section]
\newtheorem{thrm}{Theorem}

\newtheorem{prop}[thm]{Proposition}
\newtheorem{lem}[thm]{Lemma}
\newtheorem{cor}[thm]{Corollary}

\theoremstyle{remark}
\newtheorem{dfn}[thm]{Definition}
\newtheorem{rmk}[thm]{Remark}

\newtheorem{ex}{Example}
\theoremstyle{definition}

\numberwithin{equation}{section}

\DeclareMathOperator{\Sym}{Sym}

\newcommand{\Rat}{\mathbb Q}
\newcommand{\Z}{\mathbb Z}

\def\E{\mathscr{E}}

\def\P{\mathscr{P}}
\def\Pas{\mathscr{P}_{\rm as}}
\def\H{\mathscr{H}}

\def\Pt{\mathcal{P}}

\definecolor{webgreen}{rgb}{0,.4,0}
\definecolor{webbrown}{rgb}{.4,0,0}

\title{Stable Limit DAHA of type $(C^{\vee},C)$ 
and Stable Limit Koornwinder Polynomials}
\author{Dongyu Wu}
\date{}

\begin{document}

\maketitle 

\begin{abstract}
	We construct two stable limit representations of the double affine Hecke 
	algebra of type $(C^\vee,C)$ on the space of almost symmetric 
	Laurent polynomials, namely the positive and negative stable limit 
	representations. 
	Starting from the standard polynomial representation of 
	the finite rank DAHA of type $(C_n^\vee,C_n)$, 
	we study the asymptotic behavior of the Cherednik operators 
	under the two natural rescalings by positive and negative powers 
	of the parameter $t$. 
	We prove that these rescaled Cherednik operators 
	admit well-defined limits on the ring of almost symmetric 
	Laurent polynomials. 
	This yields stable positive and negative actions of 
	a common stable limit DAHA. The action of 
	the limit Cherednik operators is also proven to 
	be triangular on a natural basis of almost symmetric 
	Laurent polynomials labeled by tuple-partition symbols 
	with respect to the induced Bruhat order. 
	We further construct for 
	each of the two stable limit representations a set of 
	simultaneous eigenfunctions of the limit Cherednik operators using 
	the partial symmetrization operators acting on the 
	non-symmetric Koornwinder polynomials. We show that each of the two 
	sets of the eigenfunctions 
	form a basis of the space of almost symmetric Laurent polynomials, 
	and denote them by the positive and negative stable limit Koornwinder polynomials.
\end{abstract}

\section{Introduction}
The double affine Hecke algebras (DAHA) were first invented by I. Cherednik. 
They have been widely studied in representation theory, algebraic combinatorics, 
and mathematical physics. The standard polynomial representation of each DAHA 
gives rise to a family of orthogonal polynomials famously 
known as Macdonald-Koornwinder polynomials. For type $B,C,D,BC$ cases, 
Macdonald-Koornwinder polynomials of all other types 
can be obtained from the Koornwinder 
polynomials of type $(C_n^{\vee},C_n)$, 
a family of multivariable orthogonal polynomials 
depending on six parameters, by specialization of parameters. Therefore, 
studying DAHA of type $(C_n^{\vee},C_n)$ will help establish 
a better understanding on the standard polynomial representations of DAHA of 
all types.

A compelling direction in the Macdonald theory is to consider a stable 
limit of DAHAs as the rank tends to infinity. 
In type $A$, such limit constructions by Ion and Wu already yield 
interesting connections to double Dyck path algebra, 
and algebras of operators on them. The challenge is harder in the setting of the  
type $(C_n^{\vee},C_n)$. For 
instance, 
due to the $\mathbb{Z}_2$ reflection of the Weyl group of type $BC_n$, 
the Cherednik operators in type $(C_n^{\vee},C_n)$ are 
no longer homogeneous in the variables, and will 
involve both positive and negative powers of the variables in 
their action on the Laurent polynomials, 
which complicates the limit process. As another example, 
the partial symmetrization of the non-symmetric Koornwinder polynomials in finite rank 
will yield a family of almost symmetric Laurent polynomials in the limit which 
lies in the space of almost symmetric Laurent polynomials. 
This space has both asymptotic cyclic symmetry and asymptotic 
reflexive symmetry, which extremely complicates 
the analysis of the algebraic structure of the space 
as well as the combinatorial technicality involved in the 
stabilization process.

Ion and Wu introduced and studied the stable limit DAHA of type $A$ and developed 
its standard representation as a limit of the standard representations of 
finite rank DAHA. For the type $(C_n^{\vee},C_n)$ root system, the works 
of Noumi \cite{Nou} and Sahi \cite{Sa} established 
a standard representation of the double affine Hecke algebra of type
$(C_n^{\vee},C_n)$. Therefore it is natural to try 
to apply the framework of Ion and Wu to the type $(C_n^{\vee},C_n)$ 
setting, and to develop the associated limit theory of DAHA. By doing so, 
we are able to describe a framework for the stable lmit DAHA of type 
$(C^{\vee},C)$ in infinite rank in Section \ref{sec: DAHAsec} and 
\ref{sec: stable limit DAHA}. More specifically, 
the stable limit double affine Hecke algebra 
	$\mathscr{H}$ of type $(C^{\vee},C)$ 
	used in this paper is defined to be the  
	$\mathbb{Q}(q,t,t_0,t_{\infty},a,c)$-algebra generated by 
  $$T_0,T_1,T_2,\dots,X_1^{\pm1},X_2^{\pm1},\dots,\tilde{Y}_1,
  \tilde{Y}_2,\dots$$
  satisfying the following relations:
	  \begin{subequations}
		  \begin{equation*}
			  \begin{gathered}
			  (T_{0}-1)(T_{0}+t_0)=0, \\
			  (T_{i}-1)(T_{i}+t)=0, \quad i\geq 1,
			  \end{gathered}
			\end{equation*}
		  \begin{equation*}
			\begin{gathered}
			T_{i}T_{j}=T_{j}T_{i}, \quad |i-j|>1,\\
			T_{i}T_{i+1}T_{i}=T_{i+1}T_{i}T_{i+1}, \quad i\geq 1,\\
			T_0 T_1 T_0 T_1 = T_1 T_0 T_1 T_0,
			\end{gathered}
			\end{equation*}
		  \begin{equation*}
			  \begin{gathered}
				  t T_i^{-1} X_i T_i^{-1}=X_{i+1}, \quad i\geq 1\\
				  T_{i}X_{j}=X_{j}T_{i},\quad  j\neq i,i+1,\\
				  X_i X_j=X_j X_i \quad i,j\geq 1,\\
				  (X_1 T_0 - t_0 c^{-1})(X_1 T_0 + q^{-1}c)=0,
			  \end{gathered}
		  \end{equation*}
		  \begin{equation*}
			  \begin{gathered}
				  t T_i \tilde{Y}_i T_i=\tilde{Y}_{i+1}, \quad i\geq 1\\
				  T_{i}\tilde{Y}_{j}=\tilde{Y}_{j}T_{i},\quad  j\neq i,i+1,\\
				  \tilde{Y}_i \tilde{Y}_j=\tilde{Y}_j \tilde{Y}_i \quad i,j\geq 1,\\
				  \tilde{Y}_1 X_1 T_0 \tilde{Y}_1 =0.
			  \end{gathered}
		  \end{equation*}
	  \end{subequations}

The main difficulty of the stabilization process of the standard representation 
as usual is caused by the Cherednik operators. 
For type $(C_n^{\vee},C_n)$, the action of the Cherednik operators 
on the standard representation 
produces Laurent polynomials involving both $x_i$ and $x_i^{-1}$ with non-trivial 
coefficients. In Section \ref{sec: stable limit DAHA}, we use combinatorial techniques 
to show that by rescaling the Cherednik operators by 
positive and negative powers of $t$, these Laurent polynomials admit 
well-defined limits in the ring of almost symmetric Laurent polynomials. 
Thus, we obtain in Theorem \ref{def: positive Limit DAHA action} and 
\ref{def: negative Limit DAHA action} the positive and negative 
stable limit actions of the stable limit DAHA on the 
space of almost symmetric Laurent polynomials 
$$\Pas^{\pm} :=K[x_1^{\pm1},x_2^{\pm1},...]\otimes\Sym[X^{\pm}].$$
We state the results more precisely as follows.
\begin{thrm}
	\begin{enumerate}
		\item (Positive limit action) The stable limit double affine Hecke algebra 
		$\mathscr{H}$ of type $(C^{\vee},C)$ acts on the space $\Pas^{\pm}$ 
		of almost symmetric Laurent polynomials, with 
		the action of $Y_i$ defined as the $t$-adic limit
		$$\tilde{Y}_i = \lim_{n\rightarrow\infty}t^n Y_i^{(n)}\Pi_n$$
		and the action of $T_i,X_i^{\pm}$ defined as the classical action. 
		\item (Negative limit action) The stable limit double affine Hecke algebra 
		$\mathscr{H}$ of type $(C^{\vee},C)$ acts on the space $\Pas^{\pm}$ 
		of almost symmetric Laurent polynomials, with 
		the action of $\tilde{Y}_i$ defined as the $t^{-1}$-adic limit
		$$\tilde{Y}_i = \lim_{n\rightarrow\infty}t^{-n} Y_i^{(n)}\Pi_n$$
		and the action of $T_i,X_i^{\pm}$ defined as the classical action. 
	\end{enumerate}
\end{thrm}

Inherited from the finite rank case, 
the limit Cherednik operators still satisfy some nice 
properties. In particular, we show in Proposition \ref{prop: triangularity of positive limit Y} 
and \ref{prop: triangularity of negative limit Y} 
that the limit Cherednik operators are triangular on 
a natural basis of almost symmetric Laurent polynomials 
labeled by tuple-partition symbols with respect to the induced Bruhat order. 

\begin{thrm}
	For symbol $\lambda\vert\mu$ and $i\geq 1$, we have 
	$$\tilde{Y}_i^+ m_{\lambda\vert\mu}\in 
	\delta_i(\lambda)q^{\lambda_i}
	t^{u_{\lambda\mu}(i)}m_{\lambda\vert\mu}
	+\sum_{\nu\vert\eta<\lambda\vert\mu}\mathbb{K}'m_{\nu\vert\eta}$$
	and 
	$$\tilde{Y}_i^- m_{\lambda\vert\mu}\in 
	\delta_i'(\lambda)q^{\lambda_i}(t_0t_{\infty})
	t^{-u_{\lambda\mu}(i)}m_{\lambda\vert\mu}
	+\sum_{\nu\vert\eta<\lambda\vert\mu}\mathbb{K}'m_{\nu\vert\eta}.$$
\end{thrm}

We then proceed in Section \ref{sec: stable limit Koornwinder polynomials} 
to construct the positive and negative almost symmetric Koornwinder polynomials 
as simultaneous eigenfunctions of the limit Cherednik operators using 
the partial symmetrization operators acting on the 
non-symmetric Koornwinder polynomials in finite rank. More precisely,
\begin{align*}
	\E_{\lambda\vert\mu}^+&=
	\frac{1}{(1-t)^{l(\mu)}\prod_{i\geq 0}[m_i(\mu)]_t!}
	\lim_{n\rightarrow\infty}
	e_k^{(n)}(t_{\infty})E_{\lambda\mu 0^{n-k-l(\mu)}}\\
	\E_{\lambda\vert\mu}^-&=
	\frac{\prod_{i=1}^{r}(1-q^{\mu_{a_i}}
	t^{-\beta_{\lambda\mu}(k+a_i)-1})}
	{(1-t^{-1})^{s}\prod_{i>0}[m_i(\mu)]_{t^{-1}}!}
	\\
	&\prod_{i<j\leq r}
	\left(\frac{1-q^{-\mu_{a_i}+\mu_{a_j}}
		t^{\beta_{\lambda\mu}(k+a_i)-\beta_{\lambda\mu}(k+a_j)}}
		{1-q^{-\mu_{a_i}+\mu_{a_j}}
		t^{\beta_{\lambda\mu}(k+a_i)-\beta_{\lambda\mu}(k+a_j)+m_{a_j}(\mu)}}\right)^{m_{a_i}(\mu)}
	\lim_{n\rightarrow\infty}
	e_k^{(n)}(t_{\infty})E_{\lambda,-\mu, 0^{n-k-s}}		
\end{align*}
where $e_k^{(n)}$ is the partial symmetrization operator. We obtain 
in Theorem \ref{thm: positive stable limit Koornwinder} and 
\ref{thm: negative stable limit Koornwinder} that the positive and negative 
stable limit Koornwinder polynomials both form simultaneous eigenbases 
for the space of almost symmetric Laurent polynomials 
with respect to the positive and negative limit Cherednik operators respectively.

\begin{thrm}
	The set of stable limit Koornwinder polynomials 
	$\{\E_{\lambda\vert\mu}^{+}\}$ 
	forms a normalized simultaneous eigenbasis for $\Pas^{\pm}$ of 
	the positive stable limit Cherednik operators $\tilde{Y}_i^+$ 
	with eigenvalues $\delta_i(\lambda)q^{\lambda_i}t^{u_{\lambda\mu}(i)}$. 
	The set of stable limit Koornwinder polynomials 
	$\{\E_{\lambda\vert\mu}^{-}\}$ 
	forms a normalized simultaneous eigenbasis for $\Pas^{\pm}$ of 
	the negative stable limit Cherednik operators $\tilde{Y}_i^-$ 
	with eigenvalues 
	$\delta_i'(\lambda)(t_0t_n)q^{\lambda_i}t^{-u_{\lambda\mu}(i)}$.
\end{thrm}

By analyzing the behavior of the stable limit DAHA and 
stable limit Koornwinder polynomials and comparing with the finite-rank 
theory we are able to portrait how classical 
DAHA theory and Koornwinder-Macdonald theory behave 
or stabilize in the infinite case. The results will also 
greatly help us develop further the orthogonality, duality and 
other properties of the stable limit Koornwinder polynomials.

\section{Preliminaries}
\subsection{Symmetric functions and plethysm}
We denote by $\Sym[A]$ 
the ring of symmetric functions in an alphabet $A$. In this paper, 
we will mainly work with the alphabets $X$, $X^{-1}$ and $X^{\pm}$, 
where $X$ is the infinite alphabet $x_1,x_2,\dots$, the alphabet 
$X^{-1}$ is the infinite 
alphabet $x_1^{-1},x_2^{-1},\dots$, and $X^{\pm}$ is the infinite alphabet 
$x_1,x_1^{-1},x_2^{-1},x_2,\dots$, together with their 
modifications. The field or ring of coefficients 
$K\supseteq \Rat$ will depend on the context. For any $k\geq 1$, 
we denote by $\overline{X}_k$ the finite alphabet $x_1,x_2,\dots, x_k$ and by 
$X_k$ the infinite alphabet $x_{k+1},x_{k+2},\dots$. 
Furthermore, for any $1\leq k\leq m$, we denote by $\overline{X}_{[k,m]}$ 
the finite alphabet $x_k,\dots, x_m$. We similarly define 
$\overline{X}_k^{-1}$, $X_k^{-1}$, 
$\overline{X}_{[k,m]}^{-1}$, $\overline{X}_k^{\pm}$, 
$X_k^{\pm}$, and $\overline{X}_{[k,m]}^{\pm}$ as the corresponding modifications 
of the alphabets $X^{-1}$ and $X^{\pm}$. 
We denote by $h_n[A]$ the $n$-{th} complete symmetric functions (or polynomials) 
in $A$, by $p_n[A]$ 
the $n$-th power sum symmetric functions (or polynomials), and by $e_n[A]$
the $n$-th elementary symmetric functions (or polynomials). 
For a partition $\lambda$, $m_\lambda[A]$ denotes the monomial symmetric function 
(or polynomial) in $A$. 

Any action of the  monoid $(\Z_{>0},\cdot)$ on the ring 
$K$ extends to a canonical action by $\Rat$-algebra morphisms on $\Sym[A]$. 
The morphism corresponding to the action of $n\in \Rat_{>0}$ is denoted by 
$\mathfrak{p}_n$ 
and is defined by
$$
\mathfrak{p}_n\cdot p_k[A]=p_{nk}[A], \quad k\geq 1.
$$
The action of $(\Z_{>0},\cdot)$ on $K$ is $\Rat$-linear, 
and $\mathfrak{p}_n$ acts on 
parameters by raising them to the $n$-th power. 
For example, if $K=\Rat(q,t)$ then, 
$\mathfrak{p}_n\cdot q=q^n, ~\mathfrak{p}_n\cdot t=t^n$. 
Such parameters will be called \emph{line elements} or \emph{scalars}. 

Let $R$ be a ring with an action of $(\Z_{>0},\cdot)$ by ring morphisms. 
Any ring morphism $\varphi: \Sym[A]\to R$ that is compatible with the action 
of $(\Z_{>0},\cdot)$ is uniquely determined by the image of $p_1[A]$. 
The image of $F[A]\in \Sym[A]$ through $\varphi$ is usually denoted by 
$F[\varphi(A)]$ and called the plethystic evaluation (or substitution) of 
$F$ at $\varphi(A)$. 
As an instance, $p_1[X]$ can be regarded as the plethysm 
through the identity morphism $\varphi: p_1\mapsto p_1=x_1+x_2+\dots$. 
This is compatible with the previous definition of the symmetric functions 
and hence we also write the alphabet $X=x_1+x_2+\dots$ in the 
summation form and similarly $X^{-1} =x_1^{-1}+x_2^{-1}+\dots$ and 
$X^{\pm}=X+X^{-1}$. We remark that an element 
$F\in\Sym[X^{\pm}]$ is symmetric in the variables $x_1,x_2,...$ and also 
symmetric under the inversion $x_i\mapsto x_i^{-1}$ for each $i$. 
Equivalently, $\Sym[X^{\pm}]$ can be regarded as the ring of symmetric functions 
in the variables 
$$x_1^{\pm}=(x_1+x_1^{-1}), x_2^{\pm}=(x_2+x_2^{-1}),...$$
We also note that $h_n[X^{\pm}]$ in general is not homogeneous. 

As an example, the symmetric monomials in the alphabet $X_{k-1}^{\pm}$ can be 
expanded as follows.
\begin{ex}\label{ex: monomial expansion}
For a partition $\lambda=(\lambda_1,...,\lambda_n)$ with $\lambda_n>0$, we have 
$$m_{\lambda}[X_{k-1}^{\pm}]=
\sum_{T\subseteq S\subseteq [n]}
x_k^{2|\lambda_T|-|\lambda_S|}m_{\hat{\lambda}_S}[X_{k}^{\pm}]
$$
where $[n]=\{1,2,...,n\}$, $\lambda_S$ is the subpartition of $\lambda$ consisting of the parts 
$\lambda_i$ with $i\in S$, $\hat{\lambda}_S$ is the partition obtained from 
$\lambda$ by removing the parts $\lambda_i$ with $i\in S$. Equivalently, 
the expression can be written as
$$m_{\lambda}[X_{k}^{\pm}]=m_{\lambda}[X_{k-1}^{\pm}]-
\sum_{\substack{{T\subseteq S\subseteq [n]}\\{S\neq \emptyset}}}
x_k^{2|\lambda_T|-|\lambda_S|}m_{\hat{\lambda}_S}[X_{k}^{\pm}]
$$
\end{ex}

The plethystic exponential $\textrm{Exp}[X]$ is defined as
$$\textrm{Exp}[X]=\sum_{n=0}^{\infty}h_n[X]=
\exp\left(\sum_{n=1}^\infty \frac{p_n[X]}{n}\right).$$
It satisfies some basic properties listed below.
\begin{itemize}
	\item $\textrm{Exp}[X+Y]=\textrm{Exp}[X]\textrm{Exp}[Y]$.
	\item $\textrm{Exp}[-X]=\sum_{n=0}^{\infty}(-1)^n e_n[X]$.
	\item $\textrm{Exp}[X]=\prod_{i}(1-x_i)^{-1}$ if $X=x_1+x_2+\dots$.
	\item $\textrm{Exp}[-X]=\prod_{i}(1-x_i)$ if $X=x_1+x_2+\dots$.
\end{itemize}
The generating function of the complete symmetric functions is given by
$$\sum_{n=0}^{\infty}h_n[X] z^n = \textrm{Exp}[zX] = \prod_{i}\frac{1}{1-x_i z}.$$
By applying the formula to the alphabet 
$$(1-t)X= x_1 - t x_1 + x_2 - t x_2 + ...$$
for some parameter $t$, we also have a more generalized formula
$$\sum_{n=0}^{\infty}h_n[(1-t)X] z^n = 
\textrm{Exp}[z(1-t)X] = 
\prod_{i}\frac{1-t x_i z}{1-x_i z}.$$
We will be using this formula frequently 
in the later sections. 

\subsection{Almost symmetric functions}
Let $K$ be an ambient field. For any $k\geq 1$, we denote by
$$\iota_k: K[x_1^{\pm1},...,x_{k-1}^{\pm1}]\to 
K[x_1^{\pm1},x_2^{\pm1},...x_k^{\pm1}]$$
be the algebraic morphism defined by $\iota_k(x_i)=x_i$ for $i=1,...,k-1$. 
Then $\P^{\pm}_k=K[x_1^{\pm1},...,x_k^{\pm1}]$ forms a direct system with 
respect to the morphisms $\iota_k$. The direct limit of the system is denoted by 
$K[x_1^{\pm1},x_2^{\pm1},...]$. It is the ring consisting of all Laurent polynomials 
in finitely many variables $x_1^{\pm1},...,x_n^{\pm1}$ for some $n$ over field $K$. 
The ring of almost symmetric functions $\Pas^{\pm}$ over field $K$ is then 
defined as 
$$\Pas^{\pm} :=K[x_1^{\pm1},x_2^{\pm1},...]\otimes\Sym[X^{\pm}]$$
respectively. By convention, in $\Pas^{\pm}$
we will write $f\otimes F$ in short as $f F$ for any 
$f\in K[x_1^{\pm1},x_2^{\pm1},...]$ and $F\in \Sym[X^{\pm}]$. By definition, 
$\Pas^{\pm}$ has a natural basis 
$$\{x^{\mu} m_\lambda[X^{\pm}]\}_{k\geq 0, \mu\in \Z^k, \lambda\in \Pi}.$$
We denote $\P(k)^{\pm}$ to be the subspace of $\Pas^{\pm}$ 
consisting of all elements which are symmetric 
in the variables $x_{k+1},x_{k+2},...$ and symmetric under the inversion 
$x_i\mapsto x_i^{-1}$ for each $i>k$.

The ring of almost symmetric functions 
is also naturally equipped with an action 
of the monoid $(\Z_{>0},\cdot)$ defined in a similar way. Each 
letter $x_i$ is a line element and we have 
the equivalence 
$$X_k^{\pm} = X - x_1 - x_1^{-1} - ... - x_k - x_k^{-1}.$$
The plethysm $F[\varphi(X_k^{\pm})]$ in $\Pas^{\pm}$ is then equivalent to the plethysm 
$$F[\varphi(X - x_1 - x_1^{-1} - ... - x_k - x_k^{-1})]$$
and well-defined. 

Unlike the polynomial ring $\P^+_k=K[x_1,x_2,...,x_n]$, 
the projection from $\P^{\pm}_k$ to $\P^{\pm}_{k-1}$ 
by sending $x_k$ to zero is not a well-defined morphism. 
However, we can still define a projective system with 
respect to the linear projection 
$$\pi_k: K[x_1^{\pm1},x_2^{\pm1},...,x_{k}^{\pm1}]\to 
K[x_1^{\pm1},...,x_{k-1}^{\pm1}]$$
\[
x^{\lambda} \mapsto \begin{cases}
x^{\lambda} & \text{if } \lambda_k=0,\\
0 & \text{if } \lambda_k\neq 0.
\end{cases}
\]
Then $\P^{\pm}_k=K[x_1^{\pm1},...,x_k^{\pm1}]$ as a linear space 
forms a projective system with respect to 
the morphisms $\pi_k$. The projective limit of the system is denoted by 
$\P_\infty^{\pm}$. We naturally have the inclusion 
$\Pas^{\pm}\subseteq \P_\infty^{\pm}$ defined by 
$f\otimes F\mapsto fF$. The induced projection will de denoted by 
$$\Pi_k: \Pas^{\pm}\to K[x_1^{\pm1},x_2^{\pm1},...,x_k^{\pm1}]$$
and can be explicitly described as 
\[
\Pi_k(f(x_1,x_1^{-1},...,x_n,x_n^{-1}) F[X^{\pm}])=
\begin{cases}
f(x_1,x_1^{-1}...,x_k,x_k^{-1},0,...,0) F[\bar{X}_k^{\pm}] & \text{if } n> k,\\
f(x_1,x_1^{-1},...,x_n,x_n^{-1}) F[\bar{X}_k^{\pm}] & \text{if } n\leq k.
\end{cases}
\]
We also denote 
$$I_k: K[x_1^{\pm1},...,x_k^{\pm1}]\to \Pas^{\pm}$$
to be the algebraic morphism defined by $I_k(x_i)=x_i$ for $i=1,...,k$. 
Then $\Pas^{\pm}$ admits the following type of convergence 
proposed in \cite{IW}. 
Let $(f_k)_{k\geq 1}$ be a sequence with 
$$f_k\in \P^{\pm}_k=K[x_1^{\pm1},...,x_k^{\pm1}].$$
We say that the sequence is convergent if there exists 
$N\geq 1$ and sequences $(h_k)_{k\geq 1}$, $(g_{i,k})_{k\geq 1}$, 
$i\leq N$, $h_k, ~g_{i, k}\in \P^{\pm}_k$, and $(a_{i, 
k})_{k\leq 1}$, $i\leq N$, $a_{i, k}\in K$ such that
\begin{enumerate}
\item For any $k\geq 1$, we have $f_k=h_k+\sum_{i=1}^N a_{i, k} g_{i, k}$;
\item For any $i\leq N$, $k\geq 2$,  $\pi_k(g_{i, k})=g_{i, k-1}$ and $\pi_k(h_{k})=h_{k-1}$. 
We denote by 
$$\displaystyle g_i=\lim_{k\to \infty} g_{i,k}\quad  \text{and}\quad  \displaystyle h=\lim_{k\to \infty} h_{k}$$ 
the sequence 
$(g_{i,k})_{k\geq 
1}$ and, respectively, $(h_{k})_{k\geq 1}$ as elements of $\P_\infty^{\pm}$.  
We require that $g_i\in \Pas^{\pm}$. 
\item For any  $i\leq N$ the sequence $(a_{i, k})_{k\geq 1}$ is 
convergent in $K$. We denote  $\displaystyle a_i=\lim_{k\to \infty}(a_{i, k})$.
\end{enumerate}
If the sequence $(f_k)_{k\geq 1}$ is convergent we define its 
limit as $$\lim_k(f_k):=h+ \sum_{i=1}^N a_i g_i\in\P_\infty^{\pm}.$$
In this paper, the ambient field $K$ will always contain $\mathbb{Q}(t)$ 
as a subfield and we will always use either 
the induced $t$-adic convergence or the induced $t^{-1}$-adic convergence on $K$, 
depending on whether we are working with the positive or negative part. 
\begin{prop}\label{prop: well-defined limit}
The limit defined above is well-defined.
\end{prop}
We need the following lemma for the proof. 
\begin{lem}\label{lem: asymptotic independence}
Let $f_1,...,f_m\in \Pas^{\pm}$ be $K$-linear independent. 
Then there exists a large enough $N>0$ such that for all 
$k\geq N$, the projections $\Pi_k(f_1),...,\Pi_k(f_m)$ are 
$K$-linear independent in $\P^{\pm}_k$.
\end{lem}
\begin{proof}
Let $\{x^{\lambda}m_{\mu}[X^{\pm}]\}_{(\lambda,\mu)}$ be 
the finite set of all basis elements appearing in the expression of 
$f_1,...,f_m$ as a linear combination of the basis elements. Then 
each $f_i$ is expressed as a 
finite linear combination of the basis elements
$$f_{i}=\sum_{\lambda,\mu} a_{i,(\lambda,\mu)}
x^{\lambda}m_{\mu}[X^{\pm}]$$
Since the linear system 
$$\sum_{i=1}^{m}c_i a_{i,(\lambda,\mu)}=0$$
for all pairs $(\lambda,\mu)$ has only the zero solution, 
we can choose $m$ basis elements 
$x^{\lambda^{(j)}}m_{\mu^{(j)}}[X^{\pm}]$ such that the coefficient matrix 
$$(a_{i,j})_{1\leq i,j\leq m}=(a_{i,(\lambda^{(j)},\mu^{(j)})})_{1\leq i,j\leq m}$$
is invertible. Note that each truncation 
$x^{\lambda}m_{\mu}[X^{\pm}_k]$ for sufficiently large $k$ 
contains a distinct monomial $x^{(\lambda,0,...,0,\mu)}$. Therefore 
when $k$ is large enough, the finite rank elements 
$\Pi_k(x^{\lambda}m_{\mu}[X^{\pm}_k])$ 
are still linearly independent. Consider now such $k$ and the linear system
$$\sum_{i=1}^{m}c_i \Pi_k(f_i)=0.$$
By comparing the coefficients of the monomials 
$x^{(\lambda^{(j)},0,...,0,\mu^{(j)})}$, we have 
$\sum_{i=1}^{m}c_i a_{i,j}=0$. Since the matrix 
$(a_{i,j})_{1\leq i,j\leq m}$ is invertible, we have $c_i=0$ for all $i$.
\end{proof}
\begin{proof}[Proof of Proposition \ref{prop: well-defined limit}]
It suffices to show that the limit of the constant sequence 0 is zero, 
regardless of the auxiliary sequences chosen. For $i=1,...,m$, 
consider sequences 
$c_{i,k}\in K$ and $g_{i,k}\in \P^{\pm}_k$ such that 
$$\lim_{k\to \infty} c_{i,k}=c_i, \quad 
\lim_{k\to \infty} g_{i,k}=g_i\in \Pas^{\pm}.$$
and $$\sum_{i=1}^m c_{i,k} g_{i,k}=0 \quad \text{for any } k\geq 1.$$
Now consider the finite dimensional subspace of $\Pas^{\pm}$ 
spanned by $g_1,...,g_m$ and choose a $K$-basis $h_1,...,h_s$ for the subspace. 
Then we express each $g_i$ as a linear combination of the basis elements 
$$g_i=\sum_{j=1}^{s}b_{ij}h_j$$
By Lemma \ref{lem: asymptotic independence}, 
for large enough $k$, the projections $\Pi_k(h_1),...,\Pi_k(h_s)$ 
are linearly independent. On the other hand, we have 
$$0=\sum_{i=1}^m c_{i,k}g_{i,k}=\sum_{j=1}^s(\sum_{i=1}^m b_{ij}c_{i,k})\Pi_k(h_j)$$
Therefore, we have 
$$\sum_{i=1}^m b_{ij}c_{i,k}=0$$
for all $j$ and large enough $k$. 
Since $\lim_{k\to \infty} c_{i,k}=c_i$, we have $\sum_{i=1}^m b_{ij}c_i=0$ 
for all $j$. This yields
$$\sum_{i=1}^m c_{i}g_{i}= \sum_{j=1}^s(\sum_{i=1}^m b_{ij}c_{i})h_j=0.$$
This proves the well-definedness of the limit.
\end{proof}

We need the following properties of the convergence. The 
proof in \cite{IW} also applies to the current setting.
\begin{prop}[See \cite{IW}*{Proposition 6.21, Corollary 6.22}]\label{prop: continuity}
\begin{enumerate}
	\item Let $A_k:\P_k^{\pm}\to \P_k^{\pm}$ be a sequence of operators 
	with the property that 
	$(A_k \Pi_k f)_{k\geq 1}$ converges in $\Pas^{\pm}$ for all 
	$f\in \Pas^{\pm}$. 
	Let $A: \Pas^{\pm}\to \Pas^{\pm}$ be the limit operator of the 
	sequence $(A_k)_{k\geq 1}$ and 
	$(f_k)_{k\geq 1}$, $f_k\in \P_k^{\pm}$ be a convergent sequence such that 
	$f=\lim_{k} f_k\in \Pas^{\pm}$. Then we have 
	$$A f=\lim_{k}A_{k} f_k.$$
	\item Let $B$ be another limit operator on $\Pas^{\pm}$ of $B_k$ 
	with the same property as $A_k$. Then the operator $AB$ is the limit of the sequence 
	$(A_k B_k)_{k\geq 1}$.
\end{enumerate}
\end{prop}

\section{Double affine Hecke algebras}\label{sec: DAHAsec}
In this section, we recall 
the affine root system of type $(C_n^{\vee},C_n)$ 
and the induced Bruhat order on the weight lattice. 
It will be used to label a set of standard basis elements 
for the space of almost symmetric Laurent polynomials. 
The basis will be proved to be triangular under the action of the 
limit Cherednik operators later. Then we will recall 
the definition of the double affine Hecke algebra 
of type $(C_n^{\vee},C_n)$ and its action on the space of Laurent polynomials 
as well as prove some basic properties of the action.
\subsection{Affine root system of type $(C_n^{\vee},C_n)$}
The (non-reduced) root system of type $BC_n$ is defined 
to contain the following set of vectors 
$$\Phi_n= \{\pm \epsilon_i \pm \epsilon_j,\ 
\pm \epsilon_i,\ \pm 2\epsilon_i\}_{1\leq i< j\leq n},$$
where $\{\epsilon_i\}_{1\leq i\leq n}$ is the standard basis of 
$\mathbb{R}^n$ with the standard inner product 
$\langle \epsilon_i,\epsilon_j\rangle =\delta_{ij}$. 
$\Phi_n^+$ denotes the set of positive roots in $\Phi_n$ defined by 
$$\Phi_n^+=\{\epsilon_i \pm \epsilon_j,\ \epsilon_i,\ 2\epsilon_i\}_{1\leq i< j\leq n}.$$

The affine root system of type $(C_n^{\vee},C_n)$ is defined to be 
$$\Phi_n^{\rm{aff}} = \{\pm \epsilon_i + 
\frac{k}{2}\delta\}_{1\leq i\leq n,k \in \mathbb{Z}}
\bigcup \{\pm \epsilon_i \pm \epsilon_j + k\delta,\ 
\pm 2\epsilon_i+ k\delta\}_{1\leq i< j\leq n,k \in \mathbb{Z}},$$
where $\delta$ is the null root and 
\begin{align*}
\Phi_n^{\rm{aff},+}=&
\{\epsilon_i + \frac{k}{2}\delta\}_{1\leq i\leq n,k\geq 0}\bigcup 
\{-\epsilon_i + \frac{k}{2}\delta\}_{1\leq i\leq n,k>0}\\
\bigcup &\{\epsilon_i \pm \epsilon_j + k\delta,\ 
2\epsilon_i+ k\delta\}_{1\leq i< j\leq n,k \geq 0} 
\bigcup \{-\epsilon_i \pm \epsilon_j + k\delta,\ 
-2\epsilon_i+ k\delta\}_{1\leq i< j\leq n,k>0}
\end{align*}
the set of affine simple roots is given by 
$$\alpha_i=\epsilon_i - \epsilon_{i+1}\ (1\leq i\leq n-1),
\ \alpha_n=\epsilon_n,\ \alpha_0=\frac{1}{2}\delta-\epsilon_1.$$

The Weyl group of type $BC_n$ is generated by 
$$W_n\cong S_n\ltimes \mathbb{Z}_2^n = \langle s_{\alpha_1},...,s_{\alpha_n}\rangle$$
and the affine Weyl group $W_n^{\rm{aff}}=\mathbb{Z}^n\rtimes W_n$
of type $(C_n^{\vee},C_n)$ is generated by 
$$W_n^{\rm{aff}} = \langle s_{\alpha_0},s_{\alpha_1},...,s_{\alpha_n}\rangle$$
In short, we write $s_i$ for $s_{\alpha_i}$ for $0\leq i\leq n$. 
$W_n^{\rm{aff}}$ acts on $\mathbb{R}^n\oplus \mathbb{R}\delta$  
linearly extended from the action on the basis vectors 
$$s_i \epsilon_j = \epsilon_{s_i(j)}\ (1\leq i \leq n-1, 1\leq j \leq n),$$
$$s_n \epsilon_j=\epsilon_j\ (1\leq j \leq n-1),\quad s_n \epsilon_n=-\epsilon_n,$$ 
$$ s_0 \epsilon_1=-\epsilon_1+\delta, \quad 
s_0 \epsilon_j=\epsilon_j\ (2\leq j \leq n)$$
as well as the trivial action on $\mathbb{R}\delta$.

Let $\Lambda_n\cong \mathbb{Z}^n$ be the weight lattice. It can be 
presented as the ring 
$$P_n=\mathbb{K}[x_1^{\pm1},x_2^{\pm1},...,x_n^{\pm1}]$$
of Laurent polynomials in $n$ variables over 
$\mathbb{K}=\mathbb{Q}(q^{1/2},t^{1/2},t_0^{1/2},
t_n^{1/2},u_0^{1/2},u_n^{1/2})$. Hence $W_n^{\rm{aff}}$ naturally 
acts on $P_n$. The action is given by $w x^{\lambda}=x^{w\lambda}$ where 
$x^{\delta}=q$ and can be explicitly expressed as follows: 
$$s_i f(x_1,...,x_i,x_{i+1},...,x_n)=
f(x_1,...,x_{i+1},x_i,...,x_n), \quad (1\leq i\leq n-1)$$
$$s_n f(x_1,...,x_n)=f(x_1,...,x_n^{-1}),$$
$$s_0 f(x_1,...,x_n)=f(q^{-1} x_1^{-1},x_2,...,x_n).$$

\subsection{The Bruhat order}
The Bruhat order on $W_n^{\rm{aff}}$ is defined as the transitive closure 
of the relation 
$$w < s_\alpha w \quad \text{if} \quad w^{-1}(\alpha) \in \Phi_n^{\rm{aff},+}$$
For any $\lambda\in \Lambda_n$, we denote by $\lambda^+$ the 
unique dominant weight in the $W_n$-orbit of $\lambda$. More 
specifically, if 
$$\lambda=(\lambda_1,...,\lambda_n)$$
then 
$\lambda^+$ is the rearrangement of $|\lambda_1|,...,|\lambda_n|$ 
such that 
$$\lambda_1^+\geq \lambda_2^+\geq ...\geq \lambda_n^+.$$
Similarly, we denote by $\lambda^-$ the unique antidominant weight in the 
$W_n$-orbit of $\lambda$. It is the rearrangement of 
$-|\lambda_1|,...,-|\lambda_n|$ such that 
$$\lambda_1^-\leq \lambda_2^-\leq ...\leq \lambda_n^-.$$
Then the Bruhat order on $W_n^{\rm{aff}}$ induces a partial order on
$\Lambda_n$ defined as follows. For any $\lambda, \mu\in \Lambda_n$, 
we have
\begin{enumerate}
	\item If $\lambda^+ < \mu^+$ in the dominance order on 
	partitions, then $\lambda < \mu$ ;
	\item If $\lambda^+ = \mu^+$, then 
	$\lambda < \mu$ if $w_\lambda < w_\mu$ where 
	$w_\lambda$ and $w_\mu$ are the minimal length representatives of the 
	cosets $W_n \lambda$ and $W_n \mu$ respectively.
\end{enumerate}
The induced partial order on $\Lambda_n$ is also called the Bruhat order.

\begin{ex}
We list all vectors smaller than $(-2,0)$ below in the rank 2 space $\mathbb{Z}^2$ 
with respect to the Bruhat ordering induced from the Bruhat order in $W_2^{\rm{aff}}$ 
and arrange them increasingly.
\begin{align*}
&(0,0)<(1,0)<(0,1)<(0,-1)<(-1,0)<(1,1)<(1,-1)<(-1,1)<(-1,-1)\\
&<(2,0)<(0,2)<(0,-2)<(-2,0)
\end{align*}
\end{ex}

A strict tuple in this paper is defined to be 
an element $\lambda\in \mathbb{Z}^n$ for some $n$ 
such that $\lambda_n\neq 0$. In this case 
$\lambda$ is called a strict $n$-tuple. 
Following the notation in \cite{IW}, we use the symbol 
$\lambda\vert\mu$ for the ordered pair $(\lambda,\mu)$ 
such that $\lambda$ is a strict $k$-tuple for some $k$ and 
$\mu\in \Pi$ is a partition. Then 
the partial order on $\Lambda_n$ induces a partial order on the 
set $\Lambda^{\rm{as}}$ of all such elements $\lambda\vert\mu$ 
defined as follows
$$\lambda\vert\mu \leq \eta\vert\nu \quad \text{if} \quad 
l(\lambda)\leq l(\eta) \quad \text{and} \quad 
\lambda 0^{l(\eta)-l(\lambda)}\mu\leq \eta\nu$$
where $l$ is the length function in $W_n^{\rm{aff}}$ and 
$0^{l(\eta)-l(\lambda)}$ is the sequence of zeros of length 
$l(\eta)-l(\lambda)$. Then we have the following property of 
the Bruhat order on $\Lambda^{\rm{as}}$. 
\begin{prop}\label{prop: Bruhat order}
For any tuples $\lambda,\mu,\eta,\nu$ we have 
\begin{enumerate}
	\item If $l(\lambda)=l(\eta)$ and $\lambda\mu \leq \eta\nu$, then 
	$\lambda\mu^+ \leq \eta\nu^+$;
	\item If $\lambda,\eta$ are strict tuples, $l(\lambda)\leq l(\eta)$ 
	and $\lambda0^{l(\eta)-l(\lambda)}\mu \leq \eta\nu$, 
	then $\lambda\vert\mu^+ \leq \eta\vert\nu^+$.
\end{enumerate}
\end{prop}
We refer to \cite{IW2} for the proof for the proposition which 
holds for general root systems. 

Now for a symbol $\lambda\vert\mu$, we denote 
$$m_{\lambda\vert\mu}=x^{\lambda}m_{\mu}[X_{l(\lambda)}^{\pm}].$$
Then we have
\begin{prop}\label{prop: standard basis 1}
	$\{m_{\lambda\vert\mu}\}$ forms a basis of $\Pas^{\pm}$. As 
	a consequence, $\{m_{\lambda\vert\mu}\}_{l(\lambda)\leq k}$ 
	forms a basis of $\P(k)^{\pm}$.
\end{prop}
\begin{proof}
We first show the set spans $\Pas^{\pm}$. We will show the set spans 
all basis elements 
$\{x^{\lambda}m_{\mu}[X^{\pm}]\}$ for all $\lambda\in \mathbb{Z}^n$, 
$\mu\in \Pi$. Since 
$$x^{\lambda}m_{\mu}[X^{\pm}]=x^{\lambda}
\sum_{\nu\sqsubseteq \mu}m_{\nu}[\bar{X}_{l(\lambda)}^{\pm}]
m_{\mu\backslash\nu}[X_{l(\lambda)}^{\pm}]$$
where $\nu$ sums over all subpartitions of $\mu$ consisting 
a subset of parts of $\mu$, it suffices to show 
that for any strict tuple 
$\lambda\in \mathbb{Z}^n$, $\mu\in \Pi$ and $n\geq l(\lambda)$, 
the element $x^{\lambda}m_{\mu}[X_n^{\pm}]$ can be 
spanned as a linear combination of elements of the form $m_{\lambda\vert\mu}$. 
This can be proved 
by a double induction on the weight $|\mu|$ and $n-l(\lambda)\geq 0$. 
For $|\mu|=0$ the statement is trivial. 
For $|\mu|=1$ we have 
$$x^{\lambda}m_{1}[X_n^{\pm}]=x^{\lambda}m_{1}[X_{l(\lambda)}^{\pm}]
-x^{\lambda}(x_{l(\lambda)+1}+...+x_{n}).$$
Suppose the statement holds for all $\mu$ with $|\mu|<N$. Then for $|\mu|=N$, 
when $n=l(\lambda)$ 
the element $x^{\lambda}m_{\mu}[X_n^{\pm}]=m_{(\lambda\vert\mu)}$ 
is already in the desired form. When $n>l(\lambda)$, 
we can use the formula in Example 
\ref{ex: monomial expansion} to express $x^{\lambda}m_{\mu}[X_n^{\pm}]$ 
as a linear combination of elements satisfying the induction hypothesis 
with either a smaller weight $|\mu|$ or a smaller $n-l(\lambda)$. 

Next we show the set is linearly independent. Suppose we have a non-empty 
finite set of symbols $S=\{\lambda|\mu\}$ and a non-trivial linear relation
$$\sum_{\lambda|\mu\in S}c_{\lambda|\mu}m_{\lambda|\mu}=0$$
with $c_{\lambda|\mu}\neq 0$ for all $\lambda|\mu\in S$. 
Let $T\subset S$ be the subset containing the symbols $\lambda|\mu$ of which 
the quantity $|\lambda^+|+|\mu|$ is maximal. Let 
$\eta|\tau\in T$ be a symbol with the minimal length $l(\eta)$ 
among the symbols in $T$. Pick a large enough $N$ such that 
$N+ l(\eta)> l(\lambda)$ for all $\lambda|\mu\in S$. 
Then observe the monomial 
$$x^{\eta}\prod_{i\geq 1}x_{i+l(\eta)+N}^{\tau_i}$$
appears in the monomial expansion of $m_{\eta|\tau}$ and 
and cannot appear in the monomial expansion of any other 
$m_{\lambda|\mu}$ with $\lambda|\mu\in S$. To see this, 
note that this monomial does not appear in the monomial expansion of any 
$m_{\lambda|\mu}\in S\backslash T$ since $|\lambda^+|+|\mu|<|\eta^+|+|\tau|$. 
On the other hand, for any $m_{\lambda|\mu}\in T$ with $l(\lambda)>l(\eta)$, 
any term in the monomial expansion of $m_{(\lambda|\mu)}$ contains 
a non-zero power of some $x_i$ with $l(\eta)<i<l(\eta)+N$ and 
therefore cannot be the desired monomial. As a consequence this 
implies that $c_{\eta|\tau}=0$ and yields a contradiction. 
Therefore the set $\{m_{\lambda\vert\mu}\}$ is linearly independent.
\end{proof}

\subsection{Double affine Hecke 
algebra of type $(C_n^{\vee},C_n)$}\label{sec: DAHA}
\begin{dfn}
The Hecke algebra $H_n$ of type $BC_n$ is the 
  $\mathbb{Q}(t,t_n)$-algebra generated by $T_1,\dots,T_{n-1},T_n$ 
with the following generating relations:
\begin{subequations}\label{Hecke relations}
		\begin{equation}\label{Hecke quadratic}
			\begin{gathered}
			(T_{n}-1)(T_{n}+t_n)=0, \\
			(T_{i}-1)(T_{i}+t)=0, \quad 1\leq i\leq n-1,
			\end{gathered}
  		\end{equation}
        \begin{equation}\label{Hecke T relation ii}
          \begin{gathered}
          T_{i}T_{j}=T_{j}T_{i}, \quad |i-j|>1,\\
          T_{i}T_{i+1}T_{i}=T_{i+1}T_{i}T_{i+1}, \quad 1\leq i\leq n-2,\\
		  T_{n-1}T_n T_{n-1} T_n = T_n T_{n-1} T_n T_{n-1}.
          \end{gathered}
          \end{equation}
\end{subequations}
\end{dfn}
Let $\chi: H_n\rightarrow \mathbb{K}$ be the multiplicative character defined as 
$$\chi(T_i)=t, \quad 1\leq i\leq n-1,\quad \chi(T_n)=t_n.$$
Then there exists an idempotent $e_n\in H_n$ defined by 
$$e_n = \frac{1}{[n]_{t^{-1}}!\prod_{i=1}^{n}(1+t_n^{-1} t^{1-i})}\sum_{w\in W}\chi(T_w)^{-1}T_w.$$

To better discuss the stablization process, instead of working over $\mathbb{K}$ we set our ambient field to be 
the subfield $\mathbb{K}'=\mathbb{Q}(q,t,t_0,t_n,a,c)$ of $\mathbb{K}$, where 
$$a=t_n^{1/2}u_n^{1/2},\quad c=q^{1/2}t_0^{1/2}u_0^{1/2},$$
with some parameter shifts. More precisely, we first define the DAHA as follows.
\begin{dfn}\label{def: DAHA}
  The double affine Hecke algebra $\H_n$ of type $(C_n^{\vee},C_n)$ is the 
  $\mathbb{K}'$-algebra generated by 
$$T_0,T_1,\dots,T_{n-1},T_n,X_1^{\pm1},\dots,X_n^{\pm1}$$
satisfying the following relations:
    \begin{subequations}\label{DAHA relations}
		\begin{equation}\label{quadratic}
			\begin{gathered}
			(T_{0}-1)(T_{0}+t_0)=0, \\
			(T_{i}-1)(T_{i}+t)=0, \quad 1\leq i\leq n-1,\\
			(T_{n}-1)(T_{n}+t_n)=0,
			\end{gathered}
  		\end{equation}
        \begin{equation}\label{T relation ii}
          \begin{gathered}
          T_{i}T_{j}=T_{j}T_{i}, \quad |i-j|>1,\\
          T_{i}T_{i+1}T_{i}=T_{i+1}T_{i}T_{i+1}, \quad 1\leq i\leq n-2,\\
		  T_0 T_1 T_0 T_1 = T_1 T_0 T_1 T_0,\\
		  T_{n-1}T_n T_{n-1} T_n = T_n T_{n-1} T_n T_{n-1},
          \end{gathered}
          \end{equation}
        \begin{equation}\label{X relation ii}
            \begin{gathered}
                t T_i^{-1} X_i T_i^{-1}=X_{i+1}, \quad 1\leq i\leq n-1\\
                T_{i}X_{j}=X_{j}T_{i},\quad  \langle \alpha_i,\epsilon_j\rangle=0,\\
                X_i X_j=X_j X_i \quad 1\leq i,j\leq n,\\
				t_n X_n T_n^{-1} = T_n X_n^{-1} + (a-t_n a^{-1}),\\
				t_0 T_0^{-1} X_1^{-1} = q X_1 T_0 + (c-qt_0 c^{-1}).
            \end{gathered}
        \end{equation}
    \end{subequations}
\end{dfn}
Denote
$$Y_i = (t_0 t_n)t^{n-i}T_{i-1}...T_2 T_1 T_0^{-1} T_1^{-1}...T_n^{-1}T_{n-1}^{-1}...T_i^{-1}.$$
These elements form a commutative family and obey the following relations
\begin{subequations}\label{DAHA Y relations}
	\begin{equation}\label{Y relation ii}
		\begin{gathered}
			t^{-1} T_i Y_i T_i=Y_{i+1}, \quad 1\leq i\leq n-1\\
			T_{i}Y_{j}=Y_{j}T_{i},\quad  \langle \alpha_i,\epsilon_j\rangle=0,\\
			Y_i Y_j=Y_j Y_i \quad 1\leq i,j\leq n,\\
			 T_n^{-1} Y_n^{-1} = t_0 Y_n T_n + (t_0-1).
		\end{gathered}
	\end{equation}
\end{subequations}

The Definition \ref{def: DAHA} of DAHA 
is not the most convenient presentation for stabilization. 
In order to apply the stabilization process, 
we will be using the following alternative definition.
\begin{dfn}\label{def: Limit DAHA-alt}
	The double affine Hecke algebra $\mathscr{H}_n$ of 
	type $(C_n^{\vee},C_n)$ is alternatively defined as the 
	$\mathbb{K}'$-algebra generated by 
  $$T_0,T_1,T_2,\dots,T_{n-1},T_n,
  X_1^{\pm1},X_2^{\pm1},\dots,X_n^{\pm1},Y_1^{\pm1},Y_2^{\pm1},\dots,Y_n^{\pm1}$$
  satisfying the following relations:
	  \begin{subequations}\label{Limit DAHA relations-alt}
		  \begin{equation}\label{Limit quadratic-alt}
			  \begin{gathered}
			  (T_{0}-1)(T_{0}+t_0)=0, \\
			  (T_{i}-1)(T_{i}+t)=0, \quad 1 \leq i\leq n-1,\\
			  (T_{n}-1)(T_{n}+t_n)=0,
			  \end{gathered}
			\end{equation}
		  \begin{equation}\label{Limit T relation ii-alt}
			\begin{gathered}
			T_{i}T_{j}=T_{j}T_{i}, \quad |i-j|>1,\\
			T_{i}T_{i+1}T_{i}=T_{i+1}T_{i}T_{i+1}, \quad 1 \leq i\leq n-2,\\
			T_0 T_1 T_0 T_1 = T_1 T_0 T_1 T_0,
			\end{gathered}
			\end{equation}
		  \begin{equation}\label{Limit X relation ii-alt}
			  \begin{gathered}
				  t T_i^{-1} X_i T_i^{-1}=X_{i+1}, \quad 1 \leq i\leq n-1\\
				  T_{i}X_{j}=X_{j}T_{i},\quad  \langle \alpha_i,\epsilon_j\rangle=0,\\
				  X_i X_j=X_j X_i \quad 1 \leq i,j\leq n,\\
				  (X_1 T_0 - t_0 c^{-1})(X_1 T_0 + q^{-1}c)=0,
			  \end{gathered}
		  \end{equation}
		  \begin{equation}\label{Limit Y relation ii-alt}
			  \begin{gathered}
				  t T_i Y_i T_i=Y_{i+1}, \quad 1 \leq i\leq n-1\\
				  T_{i}Y_{j}=Y_{j}T_{i},\quad  \langle \alpha_i,\epsilon_j\rangle=0,\\
				  Y_i Y_j=Y_j Y_i \quad 1 \leq i,j\leq n,\\
				  (X_1T_0 Y_1  - t_0 a)(X_1T_0 Y_1+ t_0 t_{n}a^{-1})=0.
			  \end{gathered}
		  \end{equation}
	  \end{subequations}
  \end{dfn}

\begin{rmk}
	To see the equivalence of the two definitions of DAHA, note that 
	\begin{itemize}
		\item The relation $Y_i Y_j=Y_j Y_i$ implies the relation 
		$T_{n-1}T_n T_{n-1} T_n = T_n T_{n-1} T_n T_{n-1}$.
		\item The relation 
		$(X_1 T_0 - t_0 c^{-1})(X_1 T_0 + q^{-1}c)=0$ 
		is equivalent to the relation 
		$$t_0 T_0^{-1} X_1^{-1} = q X_1 T_0 + (c-qt_0 c^{-1}).$$
		\item The relation $(X_1T_0 Y_1  - t_0 a)(X_1T_0 Y_1+ t_0 t_{n}a^{-1})=0$ 
		is interchangeable with the relation $t_n T_n^{-1} X_n T_n^{-1} = X_n^{-1} + (a-t_n a^{-1})T_n^{-1}$ 
		by applying the relation 
		$$T_n^{-1}=(t_0 t_n)^{-1}t^{1-n}T_{n-1}...T_1 T_0 Y_1 T_1 T_2...T_{n-1}.$$
	\end{itemize}
\end{rmk}

\begin{rmk}
	Compared to the definition of DAHA in \cites{Nou,Sa}, whose generators 
	will be denoted by 
	$$\mathscr{T}_0,\mathscr{T}_1,...,
	\mathscr{T}_n,\mathscr{X}_1^{\pm1},...,\mathscr{X}_n^{\pm1}, 
	\mathscr{Y}_1^{\pm1},...,\mathscr{Y}_n^{\pm1}$$
	and whose parameters will be denoted by 
	$$\mathbf{t},\mathbf{t}_0,\mathbf{t}_n,	\mathbf{q},
	\mathbf{u}_0,\mathbf{u}_n,$$
	our definition of the double affine Hecke algebra $\mathscr{H}_n$ of 
	type $(C_n^{\vee},C_n)$ coincide with their definition 
	under the following identifications:
	$$T_0\rightarrow \mathbf{t}_0^{-1/2}\mathscr{T}_0,\quad
	T_n\rightarrow \mathbf{t}_n^{-1/2}\mathscr{T}_n,\quad
	T_i\rightarrow \mathbf{t}_i^{-1/2}\mathscr{T}_i, 
	\quad 1\leq i\leq n-1,$$
	$$X_i\rightarrow \mathscr{X}_i^{-1}, 
	\quad Y_i\rightarrow \mathscr{Y}_i^{-1},$$
	$$t\rightarrow \mathbf{t}^{-1},\quad t_0\rightarrow \mathbf{t}_0^{-1},\quad
	t_n\rightarrow \mathbf{t}_n^{-1},\quad q\rightarrow \mathbf{q}^{-1},$$
	$$a\rightarrow \mathbf{t}_n^{-1/2}\mathbf{u}_n^{1/2},\quad
	c\rightarrow \mathbf{q}^{-1/2}\mathbf{t}_0^{-1/2}\mathbf{u}_0^{1/2}.$$
\end{rmk}

\subsection{Standard representation of DAHA}
We then introduce the standard representation of the double affine Hecke algebra. 
The stabilization process will be applied to the standard representation 
to obtain the stable limit almost symmetric representation, 
which will be the main object of study in the rest of the paper. 
The existence of the stable limit representation will reduce to 
understanding the asymptotic behavior of the standard representation 
on Laurent monomials.
\begin{dfn}\label{def: standard representation}
	The standard representation of the double affine Hecke algebra $\H_n$ is 
	defined by the following action of $\H_n$ on the Laurent polynomial ring 
	$\mathbb{K}'[x_1^{\pm 1},\ldots,x_n^{\pm 1}]$: 
	$$T_i = s_i + (1-t)x_i\frac{1-s_i}{x_i-x_{i+1}}$$
	$$T_0 = s_0 + \frac{(c-qt_0 c^{-1}) x_1+(1-t_0)}{1-q x_1^2}(1-s_0)$$
	$$T_n = s_n + \frac{(a-t_n a^{-1}) x_n^{-1}+(1-t_n)}{1-x_n^{-2}}(1-s_n)$$
	$$X_i = x_i, \quad Y_i = (t_0 t_n)t^{n-i}T_{i-1}...T_2 T_1 
	T_0^{-1} T_1^{-1}...T_n^{-1}T_{n-1}^{-1}...T_i^{-1}.$$
\end{dfn}
\begin{rmk}\label{rmk: symmetry}
We note the following observation: a 
Laurent polynomial 
$$f\in \mathbb{K}'[x_1^{\pm 1},\ldots,x_n^{\pm 1}]$$
is symmetric in $x_i$ and $x_{i+1}$ for $1\leq i\leq n-1$ if and only if 
$T_i f = f$, and is symmetric in $x_n$ and $x_n^{-1}$ if and only if $T_n f = f$. 
As a simple result, 
$f$ is symmetric in $x_k,x_{k+1},...,x_n$ and also under the inversion 
$x_i\mapsto x_i^{-1}$ for some $1\leq k\leq n$ and $i=k,...,n$ 
if and only if $T_i f = f$ for all $k\leq i\leq n$.
\end{rmk}

For convenience, 
we define 
$\alpha = c-qt_0 c^{-1}$ 
and 
$\beta = a-t_n a^{-1}$. 
Then by a direct computation we have
\begin{prop}
\begin{enumerate}
	\item For any $m>0$, we have
	\begin{align*}
		T_{n}x_n^{m} &= 
		(1-t_n)(x_{n}^{m}+x_{n}^{m-2}+...+x_n^{-m}) + t_n x_{n}^{-m} \\
		&+	\beta (x_n^{m-1} + x_n^{m-3} + ... + x_n^{-(m-1)}) \\
		&= h_m[(1-t_n)x_{n}+x_n^{-1}]+\beta h_{m-1}[x_{n}+x_{n}^{-1}]
	\end{align*}
	and 
	\begin{align*}
		t_n T_{n}^{-1}x_n^{m} &= 
		(1-t_n)(x_{n}^{m-2}+x_{n}^{m-4}+...+x_n^{-m}) + t_n x_{n}^{-m} \\
		&+	\beta (x_n^{m-1} + x_n^{m-3} +  ... + x_n^{-(m-1)}) \\
		&= h_m[(1-t_n)x_{n}+x_n^{-1}]+\beta h_{m-1}[x_{n}+x_{n}^{-1}] - h_m[(1-t_n)x_{n}]
	\end{align*}
	\item For any $m>0$, we have 
	\begin{align*}
		T_{n}x_n^{-m} &= 
		(t_n-1)(x_{n}^{m-2}+x_{n}^{m-4}+...+x_n^{-m}) + t_n x_{n}^{m} \\
		&-t_n\beta (x_n^{m-1} + x_n^{m-3} + ... + x_n^{-(m-1)}) \\
		&= t_n(h_m[(1-t_n^{-1})x_{n}^{-1}+x_n]-h_m[(1-t_n^{-1})x_{n}^{-1}])
		-\beta h_{m-1}[x_{n}+x_{n}^{-1}]
	\end{align*}
	and 
	\begin{align*}
		T_{n}^{-1}x_n^{-m} &= 
		(1-t_n^{-1})(x_{n}^{m}+x_{n}^{m-2}+...+x_n^{-m}) + t_n^{-1} x_{n}^{m} \\
		&-t_n^{-1}\beta (x_n^{m-1} + x_n^{m-3} +... + x_n^{-(m-1)}) \\
		&=h_m[(1-t_n^{-1})x_{n}^{-1}+x_n]-
		t_n^{-1}\beta h_{m-1}[x_{n}+x_{n}^{-1}]
	\end{align*}
\end{enumerate}
\end{prop}

To study the stability of the Cherednik operators, 
we need to explore the action of the chain of $T$ operators 
applied to Laurent polynomials. We will first 
be working with the raising part. 
The following formula from \cite{IW} will also be useful.
\begin{prop}[\cite{IW}*{Lemma 6.31}]\label{prop: IW formula}
	For any $m>0$ and $1\leq k\leq n-1$, we have
	$$t^{n-k} T_{n-1}^{-1} ...T_{k+1}^{-1} T_{k}^{-1} x_k^{m} = 
		\sum_{i=0}^{m-1} x_{n}^{m-i} h_i[(1-t)\bar{X}_{[k,n-1]}]
		$$
\end{prop}

We prove also a similar formula for the negative monomials. 
\begin{prop}\label{prop: IW formula negative}
	For any $m>0$ and $1\leq k\leq n-1$, we have
	$$ T_{n-1}^{-1} ...T_{k+1}^{-1} T_{k}^{-1} x_k^{-m} = 
	\sum_{i=0}^{m}x_{n}^{-i} h_{m-i}
	[(1-t^{-1})\bar{X}_{[k,n-1]}^{-1}]$$
\end{prop}
\begin{proof}
	Let the generating function of $x_k^{-m}$ be defined by 
	$$\sum_{m=0}^{\infty} x_k^{-m} z^m = \frac{1}{1-x_k^{-1} z}.$$
	Then we compute as follows:
	\begin{align*}
	&T_{n-1}^{-1} ...T_{k+1}^{-1} T_{k}^{-1} \frac{1}{1-x_k^{-1} z}\\
=& T_{n-1}^{-1} ...T_{k+1}^{-1} \frac{1}{1-t^{-1} x_{k+1}^{-1} z}
\frac{1-t^{-1}x_{k}^{-1}z}{1-x_k^{-1}z} \\
=& \frac{1}{1-t^{-1} x_{n}^{-1} z}
\prod_{i=k}^{n-1}\frac{1-t^{-1}x_{i}^{-1}z}{1-x_i^{-1}z}.
	\end{align*}
	By taking the coefficients of $z^m$ we obtain 
	\begin{align*}
		T_{n-1}^{-1} ...T_{k+1}^{-1} T_{k}^{-1} x_k^{-m} =&
		h_{m}[(1-t^{-1})\bar{X}_{[k,n-1]}^{-1}+x_n^{-1}]\\
		=& \sum_{i=0}^{m} x_{n}^{-i} h_{m-i}[(1-t^{-1})\bar{X}_{[k,n-1]}^{-1}]
	\end{align*}
which proves the proposition.
\end{proof}

As a corollary, we obtain the following formula which will 
be used in the proof of the stabilization process.
\begin{cor}\label{cor: raising part}
	\begin{enumerate}
		\item 	For any $m>0$ and $1\leq k\leq n-1$, we have
		\begin{align*}
			&t_n t^{n-k} T_{n}^{-1} ...T_{k+1}^{-1} T_{k}^{-1} x_k^{m} \\
			&= h_m[(1-t)\bar{X}_{[k,n-1]}+(1-t_n)x_{n}+x_n^{-1}] 
			- h_m[(1-t)\bar{X}_{[k,n-1]}+(1-t_n)x_{n}] + \\
			&\beta h_{m-1}[(1-t)\bar{X}_{[k,n-1]}+x_{n}+x_n^{-1}]
		\end{align*}
		\item 	For any $m>0$ and $1\leq k\leq n-1$, we have
		\begin{align*}
			T_{n}^{-1} ...T_{k+1}^{-1} T_{k}^{-1} x_k^{-m} 
			=& h_m[(1-t^{-1})\bar{X}_{[k,n-1]}^{-1}+(1-t_n^{-1})x_{n}^{-1}+x_n] \\
			&- t_n^{-1}\beta h_{m-1}[(1-t^{-1})\bar{X}_{[k,n-1]}^{-1}+x_{n}+x_n^{-1}]
		\end{align*}
	\end{enumerate}
\end{cor}
\begin{proof}
	We compute directly as follows:
\begin{align*}
	&t_n t^{n-k} T_{n}^{-1} ...T_{k+1}^{-1} T_{k}^{-1} x_k^{m} \\
	&= 
	\sum_{l=0}^{m-1}t_n T_{n}^{-1} x_{n}^{m-l} h_l[(1-t)\bar{X}_{[k,n-1]}]\\
	&= h_m[(1-t)\bar{X}_{[k,n-1]}+(1-t_n)x_{n}+x_n^{-1}] 
	- h_m[(1-t)\bar{X}_{[k,n-1]}+(1-t_n)x_{n}] + \\
	&\beta h_{m-1}[(1-t)\bar{X}_{[k,n-1]}+x_{n}+x_n^{-1}]
\end{align*}
and 
\begin{align*}
	T_{n}^{-1} ...T_{k+1}^{-1} T_{k}^{-1} x_k^{-m}
	= &
	\sum_{l=0}^{m}T_{n}^{-1} x_{n}^{-l} h_{m-l}[(1-t^{-1})\bar{X}_{[k,n-1]}^{-1}]\\
	=&  h_m[(1-t^{-1})\bar{X}_{[k,n-1]}^{-1}+(1-t_n^{-1})x_{n}^{-1}+x_n] \\
	&- t_n^{-1}\beta h_{m-1}[(1-t^{-1})\bar{X}_{[k,n-1]}^{-1}+x_{n}+x_n^{-1}]
\end{align*}
which proves the corollary.
\end{proof}

\section{Stable limit DAHA of type $(C^{\vee},C)$}\label{sec: stable limit DAHA}
We now proceed proving the existence of the stable limit action. 
There are two asymptotic stabilities for the finite rank Cherednik 
operators with the positive rescaling by $t^n$ with respect to
the $t$-adic topology, and the negative rescaling by $t^{-n}$ with 
respect to the $t^{-1}$-adic topology. 
We first deal with the positive stabilization. 
We will analyze first apply the raising part of the Cherednik operators to 
Laurent monomials and 
then prove the convergence and the almost symmetric behavior 
of the resulting expressions. The negative stabilization will be 
proved in a similar way.
\subsection{Stabilization of the positive DAHA action on positive monomials}
From Corollary \ref{cor: raising part} we notice that the raising 
part of the action of the Cherednik operators on the monomial $x_k^m$ 
can be expressed as a sum of two independent parts
plethystic expressions as $\beta$ is essentially an algebraic independent parameter 
from $t$ and $t_n$. Therefore to finish the lowering part of 
the action of the Cherednik operators for stabilization, 
we need to compute separately each portion of the expression. We 
begin with monomials with positive powers. First, we have
\begin{lem}\label{lem: symmetric part 1}
	The following plethystic identity holds for all $m>0$:
	\begin{align*}
	& tT_{n-1}^{-1} (h_m[(1-t)\bar{X}_{[k,n-1]}+(1-t_n)x_{n}+x_n^{-1}] 
	- h_m[(1-t)\bar{X}_{[k,n-1]}+(1-t_n)x_{n}])\\
	=&(1-tt_n^{-1})x_{n-1}^{-1}h_{m-1}[(1-t)\bar{X}_{[k,n-2]}+(1-t t_n) x_{n-1}+x_{n-1}^{-1}+x_n+x_{n}^{-1}]\\
	&-(t_n-t)x_{n-1}^{-1}h_{m-3}[(1-t)\bar{X}_{[k,n-2]}+(1-t t_n) x_{n-1}+x_{n-1}^{-1}+x_n+x_n^{-1}]\\
	& + tt_n^{-1}x_{n-1}^{-1}h_{m-1}[(1-t)\bar{X}_{[k,n-2]}+(1-t_n)x_{n-1}+x_{n-1}^{-1}+(1-t_n)x_n+(1-t_n)x_n^{-1}]
	\end{align*}
\end{lem} 
\begin{proof}
Define the generating series 
\begin{align*}
&S_1(z) \\
=& \sum_{m\geq 0}(h_m[(1-t)\bar{X}_{[k,n-1]}+(1-t_n)x_{n}+x_n^{-1}] 
- h_m[(1-t)\bar{X}_{[k,n-1]}+(1-t_n)x_{n}])z^m\\
=& \frac{x_n^{-1}z(1-t_n x_{n}z)}{(1-x_n z)(1-x_n^{-1} z)}
\prod_{i=k}^{n-1}\frac{1-t x_i z}{1-x_i z}
\end{align*}
Then we have
\begin{align*}
tT_{n-1}^{-1} (S_1(z)\prod_{i=k}^{n-2}\frac{1-x_i z}{1-t x_i z}) 
=& tT_{n-1}^{-1} \frac{x_n^{-1}z(1-t_n x_{n}z)(1-t x_{n-1}z)}
{(1-x_n z)(1-x_n^{-1} z)(1-x_{n-1} z)} \\
=&\frac{(1-t t_n^{-1})x_{n-1}^{-1}(z- t_n z^3)(1-t t_n x_{n-1}z)}{(1-x_{n-1}z)(1-x_n z)(1-x_{n-1}^{-1}z)(1-x_n^{-1}z)}\\
& +  \frac{t t_n^{-1} x_{n-1}^{-1}z(1-t_n x_{n}z)(1-t_n x_{n}^{-1}z)(1-t_n x_{n-1}z)}
{(1-x_{n-1}z)(1-x_n z)(1-x_{n-1}^{-1}z)(1-x_n^{-1}z)}
\end{align*}
By taking the coefficients of $z^m$ we obtain the desired identity.
\end{proof}

As a corollary, we can prove the following stabilization result.
\begin{cor}\label{cor: symmetric part 1}
For $k>0$ and $m\geq 0$, let the sequence $\{a_n\}_{n\geq k}$ be defined by 
\begin{align*}
	a_n =& t^{n-k}T_{k}^{-1}...T_{n-1}^{-1}\\
	&(h_m[(1-t)\bar{X}_{[k,n-1]}+(1-t_{\infty})x_{n}+x_n^{-1}] 
	- h_m[(1-t)\bar{X}_{[k,n-1]}+(1-t_{\infty})x_{n}])
\end{align*}
where $t_{\infty}$ is a fixed parameter. 
Then there exists a limit $\lim_{n\rightarrow \infty}a_n=a\in \Pas^{\pm}$ 
for the sequence.
\end{cor}
\begin{proof}
Note first that for any $\mu\in \mathbb{Z}^{n-k}$ the coefficient of the monomial 
$x^{(0^{k-1},\mu)}$ in each $a_n$ is always a polynomial in $t_{\infty}$ and 
Laurent polynomial in $t$, 
and the degree of $t_{\infty}$ is at most $1$. Indeed, as
$$t^{n-k}T_{k}^{-1}...T_{n-1}^{-1}=(T_k+(t-1))...(T_{n-1}+(t-1))$$
and by definition $T_i = s_i + (1-t)x_i\frac{1-s_i}{x_i-x_{i+1}}$, 
we see the action of $t^{n-k+1}T_{k-1}^{-1}...T_{n-1}^{-1}$ only 
creates polynomial coefficients in $t$. On the other hand, 
the plethystic expression in the definition of $a_n$ can be expanded as:
\begin{align*}
&h_m[(1-t)\bar{X}_{[k,n-1]}+(1-t_{\infty})x_{n}+x_n^{-1}]
- h_m[(1-t)\bar{X}_{[k,n-1]}+(1-t_{\infty})x_{n}]\\
=&(h_m[(1-t)\bar{X}_{[k,n-1]}+x_{n}+x_n^{-1}]
- t_{\infty} x_n h_{m-1}[(1-t)\bar{X}_{[k,n-1]}+x_{n}+x_n^{-1}])\\
& - (h_m[(1-t)\bar{X}_{[k,n-1]}+x_{n}]
- t_{\infty} x_n h_{m-1}[(1-t)\bar{X}_{[k,n-1]}+x_{n}])
\end{align*}
which means the degree of $t_{\infty}$ in any coefficient of $a_n$ is either $0$ or $1$.

Let $S_1(z)$ be the generating series defined in Lemma \ref{lem: symmetric part 1}. 
Then by Lemma \ref{lem: symmetric part 1} we have 
{\allowdisplaybreaks
\begin{align*}
	& t^2T_{n-2}^{-1}T_{n-1}^{-1}
	(S_1(z)\prod_{i=k}^{n-3}\frac{1-x_i z}{1-t x_i z})\\
	=& tT_{n-2}^{-1}(\frac{1-t x_{n-2}z}{1- x_{n-2}z}
	(\frac{(1-t t_{\infty}^{-1})x_{n-1}^{-1}z(1- t_{\infty} z^2)(1-t t_{\infty} x_{n-1}z)}
	{(1-x_{n-1}z)(1-x_{n-1}^{-1}z)(1-x_n z)(1-x_n^{-1}z)}\\
	& +  \frac{t t_{\infty}^{-1} x_{n-1}^{-1}z(1-t_{\infty} x_{n}z)(1-t_{\infty} x_{n}^{-1}z)(1-t_{\infty} x_{n-1}z)}
	{(1-x_{n-1}z)(1-x_{n-1}^{-1}z)(1-x_n z)(1-x_n^{-1}z)}))\\
	=&\frac{(1-t t_{\infty}^{-1})(1-t_{\infty}^{-1})
	x_{n-2}^{-1}z(1- t t_{\infty} z^2)(1- t_{\infty} z^2)(1-t^2 t_{\infty} x_{n-2}z)}
	{(1-x_{n-2}z)(1-x_{n-2}^{-1}z)}\\
	&\frac{1}{(1-x_{n-1}z)(1-x_{n-1}^{-1}z)(1-x_n z)(1-x_n^{-1}z)}\\
	&+\frac{t_{\infty}^{-1}(1-t t_{\infty}^{-1})x_{n-2}^{-1}z(1- t_{\infty} z^2)(1-t t_{\infty} x_{n-2}z)}
	{(1-x_{n-2}z)(1-x_{n-2}^{-1}z)}\\
	&\frac{(1-t)(1-t t_{\infty}^2 z^2)}
	{(1-x_{n-1}z)(1-x_{n-1}^{-1}z)(1-x_n z)(1-x_n^{-1}z)}\\
	&+\frac{tt_{\infty}^{-1}(1-t t_{\infty}^{-1})x_{n-2}^{-1}z(1- t_{\infty} z^2)(1-t t_{\infty} x_{n-2}z)}
	{(1-x_{n-2}z)(1-x_{n-2}^{-1}z)}\\
	&\frac{(1-t_{\infty} x_{n}z)(1-t_{\infty} x_{n}^{-1}z)+
	(1-t_{\infty} x_{n-1}z)(1-t_{\infty} x_{n-1}^{-1}z)}
	{(1-x_{n-1}z)(1-x_{n-1}^{-1}z)(1-x_n z)(1-x_n^{-1}z)}\\
	&+\frac{t^2t_{\infty}^{-2}x_{n-2}^{-1}z(1-t_{\infty} x_{n-2}z)}
	{(1-x_{n-2}z)(1-x_{n-2}^{-1}z)}
	\frac{(1-t_{\infty} x_{n-1}z)(1-t_{\infty} x_{n-1}^{-1}z)(1-t_{\infty} x_{n}z)(1-t_{\infty} x_{n}^{-1}z)}
	{(1-x_{n-1}z)(1-x_{n-1}^{-1}z)(1-x_n z)(1-x_n^{-1}z)}\\
\end{align*}}

This expression is symmetric in $x_{n-1},x_{n-1}^{-1}$ and $x_n,x_n^{-1}$. 
Hence the action of\\
$t^{n-k}T_{k}^{-1}...T_{n-1}^{-1}$ on the plethystic 
expression in the definition of $a_n$ yields a symmetric Laurent polynomial in 
$x_{i},x_{i}^{-1}$ for $k+1\leq i\leq n$. It remains to check 
the convergence of all the coefficients. Since $a_n$ 
is symmetric $x_{i},x_{i}^{-1}$ for $k\leq i\leq n$, it can 
be written as a linear combination of almost symmetric polynomials 
$$a_n = \sum_{r,\lambda}c_{r,\lambda} x_{k}^{r}
m_{\lambda}[\bar{X}_{[k,n]}^{\pm}],$$
summed over all $|r|\leq m$ and $|\lambda|\leq m$, 
with each $c_{r,\lambda}$ being a polynomial in $t_{\infty}$ and $t$. It 
suffices to prove the stability of each coefficient $c_{r,\lambda}$ 
in front of the monomial $x_{k}^{r}x^{(0^{k-1},\lambda)}$ when 
$n$ is large enough. 

To show this, let the generating series $S_2$ be defined as
	\begin{align*}
		S_2(z) =& \sum_{m\geq 0}
		(h_m[(1-t)\bar{X}_{[k,n-1]}+(1-t_{\infty})x_{n}+x_n^{-1}] \\
		&- h_m[(1-t)\bar{X}_{[k,n-1]}+(1-t_{\infty})x_{n}])z^{m}\\
		=& \frac{x_n^{-1}z(1-t_{\infty}x_n z)}{(1-x_n z)(1-x_n^{-1}z)}
		\prod_{i=k}^{n-1}\frac{1-t x_i z}{1-x_i z}
	\end{align*}
	Note that in the statement of \ref{lem: symmetric part 1}, 
	$t_{\infty}$ is a free parameter and 
	therefore can be replaced by $t_{\infty}t^i$ for arbitrary $i$. 
	Hence by \ref{lem: symmetric part 1} we have
{\allowdisplaybreaks
\begin{align*}
	& t^{n-k}T_{k}^{-1}...T_{n-1}^{-1}S_2(z)\\
	=&(\frac{(1-t t_{\infty}^{-1})x_{n-1}^{-1}z(1- t_{\infty} z^2)(1-t t_{\infty} x_{n-1}z)}
	{(1-x_{n-1}z)(1-x_{n-1}^{-1}z)}
	\frac{1}{(1-x_n z)(1-x_n^{-1}z)}\\
	& +  \frac{t t_{\infty}^{-1} x_{n-1}^{-1}z(1-t_{\infty} x_{n-1}z)}
	{(1-x_{n-1}z)(1-x_{n-1}^{-1}z)}
	\frac{(1-t_{\infty} x_{n}z)(1-t_{\infty} x_{n}^{-1}z)}{(1-x_n z)(1-x_n^{-1}z)})
	\frac{\prod_{i=k}^{n-2}(1-t x_i z)}{\prod_{i=k}^{n-2}(1-x_i z)}\\
	=&\sum_{\substack{{\epsilon=\{\epsilon_0,...,\epsilon_{n-k-1}\}}
	\\{\epsilon_i\in \{0,1\}}}}
	\frac{x_k^{-1}z}{(1-x_{k} z)(1-x_{k}^{-1}z)}\\
	&\prod_{j=0}^{n-k-1}(\frac{(1-\bar{\epsilon}_j t_{\infty}
	 t^{\sum_{l=0}^{j-1}\epsilon_l}	 x_{n-j}z)
	(1-\bar{\epsilon}_j t_{\infty} t^{\sum_{l=0}^{j-1}\epsilon_l}x_{n-j}^{-1}z)}
	{(1-x_{n-j} z)(1-x_{n-j}^{-1}z)}\\
	&\prod_{j=0}^{n-k-1}(1-t_{\infty}^{-1}t^{1+\sum_{l=0}^{j-1}\epsilon_l} -
	(1-2t_{\infty}^{-1} t^{1+\sum_{l=0}^{j-1}\epsilon_l} )\epsilon_j)
	(1-t_{\infty}  t^{\sum_{l=0}^{j-1}\epsilon_l} z^2)^{\epsilon_j})
\end{align*}}
where $\bar{\epsilon}_j=1-\epsilon_j$. First note 
that $x_{k}^{r}x^{(0^{k-1},\lambda)}$ does not 
contain any variable $x_{n-j}$ for $n-j-k+1>l(\lambda)$, and therefore 
up to a finite number of terms in the generating series, 
we only need to consider the convergence the constant terms that may 
contribute to the coefficient of $x_{k}^{r}x^{(0^{k-1},\lambda)}$ 
in the generating series. Since $m$ is fixed and finite, 
the coefficient of $z^m$ in the generating series can only take a 
finite sum of nontrivial constant terms from 
$$\prod(1-\bar{\epsilon}_j t_{\infty}
t^{\sum_{l=0}^{j-1}\epsilon_l} x_{n-j}z)
(1-\bar{\epsilon}_j t_{\infty} t^{\sum_{l=0}^{j-1}\epsilon_l}x_{n-j}^{-1}z)
(1-t_{\infty}  t^{\sum_{l=0}^{j-1}\epsilon_l} z^2)^{\epsilon_j}$$
for $n-j-k+1>l(\lambda)$. The contribution to the 
degree of $t_{\infty}$ from these terms are hence bounded above. 
On the other hand, the degree of $t_{\infty}$ in each coefficient of $a_n$ 
is non-negative and at most 1. Therefore, the coefficient of the monomial 
$x_{k}^{r}x^{(0^{k-1},\lambda)}$ in $a_n$ can only take a 
finite sum of nontrivial terms from 
$$\prod(1-t_{\infty}^{-1}t^{1+\sum_{l=0}^{j-1}\epsilon_l} -
	(1-2t_{\infty}^{-1} t^{1+\sum_{l=0}^{j-1}\epsilon_l} )\epsilon_j)$$
Finally, observe the fact that for each finite $N>0$, the coefficient of $t_{\infty}^N$
\begin{align*}
	&\lim_{n\rightarrow\infty}[t_{\infty}^N]
	\sum_{\epsilon}\prod(1-t_{\infty}^{-1}t^{1+\sum_{l=0}^{j-1}\epsilon_l} -
		(1-2t_{\infty}^{-1} t^{1+\sum_{l=0}^{j-1}\epsilon_l} )\epsilon_j)\\
	=&\lim_{n\rightarrow\infty}[t_{\infty}^N]
	\sum_{\epsilon}\prod(1-t_{\infty}^{-1}t^{1+\sum_{l=0}^{j-1}\epsilon_l})^{1-\epsilon_j}
	(t_{\infty}^{-1}t^{1+\sum_{l=0}^{j-1}\epsilon_l})^{\epsilon_j}\\
=&\lim_{n\rightarrow\infty}\sum_{l(\epsilon)<n-k}
(-1)^{\sum \bar{\epsilon}_j}t^{N+\sum_{j}\sum_{l<j}\epsilon_{l}}
\end{align*}
is convergent. Therefore the coefficient of $x_{k}^{r}x^{(0^{k-1},\lambda)}$ 
in $a_n$ is also convergent as $n\rightarrow \infty$.
\end{proof}

Then we need to control the second portition of the resulting Laurent 
polynomials occurred in Corollary \ref{cor: raising part}, 
namely the term proportional to $\beta$. We then 
need the following plethystic identity.
\begin{lem}\label{lem: symmetric part 2}
For $k>0$ and $m\geq 1$, we have 
\begin{align*}
	&t^{n-k} T_{k}^{-1}...T_{n-1}^{-1} 
	h_{m-1}[(1-t)\bar{X}_{[k,n-1]}+x_{n}+x_n^{-1}]\\
=& h_{m-1}[x_{k}+x_{k}^{-1}+(1-t)\bar{X}_{[k+1,n]}^{\pm}]
-\sum_{i=0}^{\lfloor m/2\rfloor}(t^{i}h_{m-1-2i}[x_k+(1-t)\bar{X}_{[k+1,n]}^{\pm}]
-t^{n-k+i})
\end{align*}
\end{lem}
\begin{proof}
We first compute directly as follows.
\begin{align*}
	& t^{n-k} T_{k}^{-1}...T_{n-1}^{-1} 
	h_{m-1}[(1-t)\bar{X}_{[k,n-1]}+x_{n}+x_n^{-1}] \\
	=& t^{n-k} T_{k}^{-1}...T_{n-1}^{-1} 
	\sum_{l=0}^{m-1} h_l[-t x_{n-1}+x_{n}^{-1}]
	h_{m-l-1}[(1-t)\bar{X}_{[k,n-2]}+(x_{n}+x_{n-1})] \\
	=& t^{n-k} T_{k}^{-1}...T_{n-1}^{-1} 
	(\sum_{l=1}^{m-1} (1-t x_n x_{n-1})x_n^{-l}
	h_{m-l-1}[(1-t)\bar{X}_{[k,n-2]}+(x_{n}+x_{n-1})]\\
	&+h_{m-1}[(1-t)\bar{X}_{[k,n-2]}+(x_{n}+x_{n-1})])\\
	=& t^{n-k-1} T_{k}^{-1}...T_{n-2}^{-1}
	\sum_{l=0}^{m-1}h_{m-l-1}[(1-t)\bar{X}_{[k,n-2]}+(x_{n}+x_{n-1})]\\
	& ((h_l[(1-t)x_{n}^{-1}+x_{n-1}^{-1}]-(1-t)x_{n}^{-l})
	-tx_n h_{l-1}[(1-t)x_{n}^{-1}+x_{n-1}^{-1}])\\
	=& t^{n-k-1} T_{k}^{-1}...T_{n-2}^{-1}
	 (h_{m-1}[(1-t)\bar{X}_{[k,n-2]}+x_{n-1}+x_{n-1}^{-1}+(1-t)(x_{n}+x_n^{-1})]\\
	&- (1-t)h_{m-1}[(1-t)\bar{X}_{[k,n-2]}+x_{n-1}+x_n+x_n^{-1}])
\end{align*}
Note that
$$T_{k}^{-1}...T_{n-2}^{-1}h_{m-1}[(1-t)\bar{X}_{[k,n-2]}
+x_{n-1}+x_n+x_n^{-1}]= t^{n-k-1}h_{m-1}[x_{k}+x_n+x_n^{-1}].$$
Thus we have 
\begin{align*}
	& t^{n-k} T_{k}^{-1}...T_{n-1}^{-1} 
	h_{m-1}[(1-t)\bar{X}_{[k,n-1]}+x_{n}+x_n^{-1}] \\
	=& h_{m-1}[x_{k}+x_{k}^{-1}+(1-t)\bar{X}_{[k+1,n]}^{\pm}]\\
	&-(1-t)\sum_{i=0}^{n-k-1}t^{n-k-i-1}h_{m-1}
	[x_{k}+x_{n-i}+x_{n-i}^{-1}+(1-t)\sum_{j=0}^{i-1}
	(x_{n-j}+x_{n-j}^{-1})]
\end{align*}

Now define the generating series
$$
	S_3(z)=\sum_{m\geq 1}\sum_{i=k+1}^{n}t^{i-k-1}h_{m-1}
	[x_k+x_{i}+x_{i}^{-1}+(1-t)\bar{X}_{[i+1,n]}]z^{m-1}
$$

Then we have
\begin{align*}
S_3(z)(1-x_k z)
=&\sum_{i=k+1}^{n} \frac{t^{i-k-1}}{(1-x_i z)(1-x_i^{-1} z)}
\prod_{j=i+1}^{n}\frac{(1-t x_j z)(1-t x_j^{-1} z)}{(1-x_j z)(1-x_j^{-1} z)} \\
=&\frac{1}{(1-t)(1-tz^2)}\sum_{i=k+1}^{n}(t^{i-k-1}\prod_{j=i}^{n}\frac{(1-t x_j z)(1-t x_j^{-1} z)}{(1-x_j z)(1-x_j^{-1} z)}\\
&-t^{i-k}\prod_{j=i+1}^{n}\frac{(1-t x_j z)(1-t x_j^{-1} z)}{(1-x_j z)(1-x_j^{-1} z)})\\
=&\frac{1}{(1-t)(1-tz^2)}(\prod_{j=k+1}^{n}\frac{(1-t x_j z)(1-t x_j^{-1} z)}{(1-x_j z)(1-x_j^{-1} z)}-t^{n-k})
\end{align*}
By expanding the series 
we obtain the statement in the lemma.
\end{proof}

Now we are ready to prove the stabilization of the DAHA action 
on finite monomials. More precisely, we have the following result.
\begin{prop}\label{prop: stable limit of positive Y action 1}
Let $x^{\mu}$ be a monomial with $\mu\in \mathbb{Z}^k_{\geq 0}$ for some 
$k\geq 0$. Then for each fixed $i$, the sequence of Laurent polynomials 
$$\{t^n Y_i^{(n)}x^{\mu}\}_{n\geq \max\{i,k\}}$$
is convergent to an almost symmetric Laurent polynomial in $\Pas^{\pm}$ 
as $n\rightarrow \infty$, where $Y_i^{(n)}$ is the $Y_i$ element in the double 
affine Hecke algebra $\H_n$ of type $(C_n^{\vee}, C_n)$ with the parameters 
$(q,t,t_0,t_{\infty},a,c)$.
\end{prop}
\begin{proof}
	First suppose $i\leq k$, we have 
	$$t^n Y_i^{(n)}x^{\mu}=(t_0 t_{\infty})t^{2n-k}
	T_{i-1}...T_1 T_0^{-1}T_1^{-1}...(T_n^{(n)})^{-1}T_{n-1}...T_k^{-1}
	(t^{k-i}T_{k-1}^{-1}...T_i^{-1}x^{\mu})$$
	Note that the action of $T_0,...,T_{n-1}$ 
	are stable for rank greater than or equal to $n$ and hence 
	we only indicate the rank for $T_n^{(n)}$ for brevity. Now that 
	$$(t^{k-i}T_{k-1}^{-1}...T_i^{-1}x^{\mu})$$
	is a polynomial in $x_1,...,x_k$ and 
	can be written as a linear combination of monomials as
	$$(t^{k-i}T_{k-1}^{-1}...T_i^{-1}x^{\mu})=\sum_{\nu\in\mathbb{Z}^{k-1}_{\geq 0},m\geq 0}
	c_{\nu,m}(t)x^{\nu}x_k^{m}$$
	where each $c_{\nu,m}(t)$ is a polynomial in $t$ with integer coefficients. 
	Hence it suffices to show the convergence of each monomial 
	$x^{\nu}x_k^{m}$. We can further assume $m>0$, as for $m=0$ it is 
	not difficult to see that the limit of the action on the monomial is 0. 
	Then by Corollary \ref{cor: raising part} we have 
	\begin{align*}
	&t_{\infty}t^{n-k}(T_n^{(n)})^{-1}T_{n-1}...T_k^{-1}x^{\nu}x_k^{m}\\
	&= x^{\nu}(h_m[(1-t)\bar{X}_{[k,n-1]}+(1-t_n)x_{n}+x_n^{-1}] 
	- h_m[(1-t)\bar{X}_{[k,n-1]}+(1-t_n)x_{n}] \\
	&+\beta h_{m-1}[(1-t)\bar{X}_{[k,n-1]}+x_{n}+x_n^{-1}])
	\end{align*}
	By Corollary \ref{cor: symmetric part 1} and Lemma \ref{lem: symmetric part 2} 
	we have that the action of $t^{n-k}T_{k}^{-1}...T_{n-1}^{-1}$ on 
	both 
	$$x^{\nu}(h_m[(1-t)\bar{X}_{[k,n-1]}+(1-t_n)x_{n}+x_n^{-1}] 
	- h_m[(1-t)\bar{X}_{[k,n-1]}+(1-t_n)x_{n}])$$
	and 
	$$x^{\nu}h_{m-1}[(1-t)\bar{X}_{[k,n-1]}+x_{n}+x_n^{-1}]$$
	are convergent as $n\rightarrow \infty$. Write the limit function in 
	$\Pas^{\pm}(k)$ as a linear combination of the natural basis
	$$t_{\infty}t^{n-k}(T_n^{(n)})^{-1}T_{n-1}...T_k^{-1}x^{\nu}x_k^{m}=
	\sum_{\substack{{\sigma,\lambda}\\{l(\sigma)\leq k}}}
	C(\sigma,\lambda,\nu,m)x^{\sigma}m_{\lambda}[X^{\pm}]\in\Pas^{\pm}(k)$$
	Then we see the sequence $t^n Y_i^{(n)}x^{\mu}$ converges to 
	\begin{align*}
		&t^k T_{i-1}...T_1 T_0^{-1}T_1^{-1}...T_{k-1}^{-1}
		\sum_{\substack{{\nu,m,\sigma,\lambda}\\{l(\sigma)\leq k}}}c_{\nu,m}C(\sigma,\lambda,\nu,m)
		x^{\sigma}m_{\lambda}[X^{\pm}]\\
		=&\sum_{\substack{{\nu,m,\sigma,\lambda}\\{l(\sigma)\leq k}}}
		t^k T_{i-1}...T_1 T_0^{-1}(m_{\lambda}[X^{\pm}]T_1^{-1}...T_{k-1}^{-1}
		c_{\nu,m}C(\sigma,\lambda,\nu,m)x^{\sigma})
	\end{align*}
	which is an element in $\Pas^{\pm}(k)$. This 
	proves the convergence of the sequence $t^n Y_i^{(n)}x^{\mu}$ when $i\leq k$. 
	For $i>k$ it is easy to see the sequence $t^n Y_i^{(n)}x^{\mu}$ 
	converges to zero. 
\end{proof}

\subsection{Stabilization of the positive DAHA action on negative monomials}
Now we analyze the action of the positively rescaled Cherednik 
operators on the negative monomials. We 
still compute term by term. First we have
\begin{lem}\label{lem: symmetric part 3}
	For $k>0$ and $m\geq 0$, we have 
	\begin{align*}
		&tT_{n-1}^{-1}h_m[(1-t^{-1})\bar{X}_{[k,n-1]}^{-1}+(1-t_n^{-1})x_{n}^{-1}+x_n]\\
		=& t_n h_m[(1-t^{-1})\bar{X}_{[k,n-2]}^{-1}+x_{n-1}+(1-t_n^{-1})x_{n-1}^{-1}+
		(1-t_n^{-1})x_n+(1-t_n^{-1})x_n^{-1}]\\
		& + (t-t_n)h_m[(1-t^{-1})\bar{X}_{[k,n-2]}^{-1}+x_{n-1}+
		(1-t^{-1}t_n^{-1})x_{n-1}^{-1}+x_n+x_n^{-1}]\\
		& + (1-t t_n^{-1})h_{m-2}[(1-t^{-1})\bar{X}_{[k,n-2]}^{-1}+
		x_{n-1}+(1-t^{-1}t_n^{-1})x_{n-1}^{-1}+x_n+x_n^{-1}].
	\end{align*}
\end{lem}
\begin{proof}
	Define the generating series 
	\begin{align*}
		S_4(z)=&\sum_{m\geq 0}
	h_m[(1-t^{-1})\bar{X}_{[k,n-1]}^{-1}+(1-t_n^{-1})x_{n}^{-1}+x_n]z^m	\\
	=& \frac{1-t_n^{-1}x_{n}^{-1}z}{(1-x_n z)(1-x_n^{-1}z)}
	\prod_{i=k}^{n-1}\frac{1-t^{-1}x_i^{-1}z}{1-x_i^{-1}z}
	\end{align*}
	Then by direct computation we have
	\begin{align*}
		tT_{n-1}^{-1}
		(S_4(z)\prod_{i=k}^{n-2}\frac{1-x_i^{-1}z}{1-t^{-1}x_i^{-1}z})
		=& tT_{n-1}^{-1}(\frac{1-t_n^{-1}x_{n}^{-1}z}{(1-x_n z)(1-x_n^{-1}z)}
		\frac{1-t^{-1}x_{n-1}^{-1}z}{1-x_{n-1}^{-1}z})\\
		=& \frac{t_n(1-t_n^{-1}x_{n-1}^{-1}z)(1-t_n^{-1}x_n z)(1-t_n^{-1}x_n^{-1}z)}
		{(1-x_{n-1}z)(1-x_{n-1}^{-1}z)(1-x_n z)(1-x_n^{-1}z)}\\
		&+\frac{(1-t^{-1}t_n^{-1}x_{n-1}^{-1}z)((t-t_n)+(1-t t_n^{-1})z^2)}
		{(1-x_{n-1}z)(1-x_{n-1}^{-1}z)(1-x_n z)(1-x_n^{-1}z)}
	\end{align*}
	By taking the coefficient of $z^m$ we obtain the desired formula.
\end{proof}
As a corollary we have
\begin{cor}\label{cor: symmetric part 3}
	For $k>0$ and $m\geq 0$, 
	let the sequence $\{a_n\}_{n\geq k}$ be defined by 
	$$a_n=t^{2n-k}T_k^{-1}...T_{n-1}^{-1}
	h_m[(1-t^{-1})\bar{X}_{[k,n-1]}^{-1}+(1-t_{\infty}^{-1})x_{n}^{-1}+x_n]$$
where $t_{\infty}$ is a fixed parameter. Then each $a_n$ 
is a symmetric Laurent polynomial in $x_{i},x_{i}^{-1}$ for $k+1\leq i\leq n$, 
and the sequence $\{a_n\}_{n\geq k}$ is convergent to 0 as 
$n\rightarrow \infty$.
\end{cor}
\begin{proof}
	Let $S_4(z)$ be the generating series defined in 
	Lemma \ref{lem: symmetric part 3} with $t_n$ replaced by $t_{\infty}$. 
	Then by Lemma \ref{lem: symmetric part 3} we have
	{\allowdisplaybreaks
	\begin{align*}
		& t^2T_{n-2}^{-1}T_{n-1}^{-1}
		(S_4(z)\prod_{i=k}^{n-3}\frac{1-x_i^{-1}z}{1-t^{-1}x_i^{-1}z})\\
		=& tT_{n-2}^{-1}(\frac{1-t^{-1}x_{n-2}^{-1}z}{1-x_{n-2}^{-1}z}
		(\frac{t_{\infty}(1-t_{\infty}^{-1}x_{n-1}^{-1}z)(1-t_{\infty}^{-1}x_n z)
		(1-t_{\infty}^{-1}x_n^{-1}z)}
		{(1-x_{n-1}z)(1-x_{n-1}^{-1}z)(1-x_n z)(1-x_n^{-1}z)}\\
		&+\frac{(1-t^{-1}t_{\infty}^{-1}x_{n-1}^{-1}z)((t-t_{\infty})+
		(1-t t_{\infty}^{-1})z^2)}
		{(1-x_{n-1}z)(1-x_{n-1}^{-1}z)(1-x_n z)(1-x_n^{-1}z)}))\\
		=& \frac{t_{\infty}^2(1-t_{\infty}^{-1}x_{n-2}^{-1}z)}
		{(1-x_{n-2}z)(1-x_{n-2}^{-1}z)}
		\frac{(1-t_{\infty}^{-1}x_n z)(1-t_{\infty}^{-1}x_n^{-1}z)
		(1-t_{\infty}^{-1}x_{n-1} z)(1-t_{\infty}^{-1}x_{n-1}^{-1}z)}
		{(1-x_n z)(1-x_n^{-1}z)(1-x_{n-1} z)(1-x_{n-1}^{-1}z)}\\
		& + \frac{t_{\infty}((t-t_{\infty})+
		(1-t t_{\infty}^{-1})z^2)(1-t^{-1}t_{\infty}^{-1}x_{n-2}^{-1}z)}
		{(1-x_{n-2}z)(1-x_{n-2}^{-1}z)}\\
		&\frac{(1-t_{\infty}^{-1}x_n z)(1-t_{\infty}^{-1}x_n^{-1}z)
		+(1-t_{\infty}^{-1}x_{n-1} z)(1-t_{\infty}^{-1}x_{n-1}^{-1}z)
		+(t-1)(1-t^{-1}z^2)}
		{(1-x_{n-1} z)(1-x_{n-1}^{-1}z)(1-x_n z)(1-x_n^{-1}z)}\\
		& + \frac{((t-t t_{\infty})+
		(1-t_{\infty}^{-1})z^2)((t-t_{\infty})+
		(1-t t_{\infty}^{-1})z^2)(1-t^{-2}t_{\infty}^{-1}x_{n-2}^{-1}z)}
		{(1-x_{n-2}z)(1-x_{n-2}^{-1}z)}\\
		& \frac{1}
		{(1-x_n z)(1-x_n^{-1}z)(1-x_{n-1} z)(1-x_{n-1}^{-1}z)}
	\end{align*}}
This expression is symmetric in $x_{n-1},x_{n-1}^{-1}$ and $x_n,x_n^{-1}$. 
Hence each $a_n$ is symmetric in $x_{n-1},x_{n-1}^{-1}$ and $x_n,x_n^{-1}$. 
By the same argument we obtain that 
each $a_n$ is a symmetric Laurent polynomial 
in $x_{i},x_{i}^{-1}$ for all $k+1\leq i\leq n$. 

Note now that $m$ is fixed, hence each monomial in the expression of
$$h_m[(1-t^{-1})\bar{X}_{[k,n-1]}^{-1}+(1-t_{\infty}^{-1})x_{n}^{-1}+x_n]$$
has the coefficient with degree at least $-m$ in $t$ and 
degree either 0 or $-1$ in $t_{\infty}$. Again, as
$$t^{n-k}T_{k}^{-1}...T_{n-1}^{-1}=(T_k+(t-1))...(T_{n-1}+(t-1))$$
and by the definition of $T_i$, 
we see the action of $t^{n-k+1}T_{k-1}^{-1}...T_{n-1}^{-1}$ only 
creates polynomial coefficients in $t$. Therefore the coefficient of 
each monomial in the expression of 
$$t^{n-k}T_k^{-1}...T_{n-1}^{-1}
	h_m[(1-t^{-1})\bar{X}_{[k,n-1]}^{-1}+(1-t_{\infty}^{-1})x_{n}^{-1}+x_n]$$
has the coefficient with degree at least $-m$ in $t$ and 
degree either 0 or $-1$ in $t_{\infty}$. Now that 
$$a_n = t^n(t^{n-k}T_k^{-1}...T_{n-1}^{-1}
h_m[(1-t^{-1})\bar{X}_{[k,n-1]}^{-1}+(1-t_{\infty}^{-1})x_{n}^{-1}+x_n])$$
We see clearly $a_n$ is convergent to 0 as 
$n\rightarrow \infty$.
\end{proof}

Next we have
\begin{lem}\label{lem: symmetric part 4}
	For $k>0$ and $m\geq 1$, we have 
	\begin{align*}
		&t^{n-k}T_k^{-1}...T_{n-1}^{-1}
		h_{m-1}[(1-t^{-1})\bar{X}_{[k,n-1]}^{-1}+x_{n}+x_n^{-1}]\\
		=& h_{m-1}[x_k+x_k^{-1}]-
		\sum_{i=0}^{\lfloor (m-1)/2 \rfloor}
		(t^{-i}h_{m-2i-1}[x_k+(1-t^{-1})x_{k}^{-1}]\\
		&-t^{n-k-i}h_{m-2i-1}[x_k+(1-t^{-1})x_k^{-1}
		+(1-t^{-1})\bar{X}_{[k+1,n]}^{-1}])
	\end{align*}
\end{lem}
\begin{proof}
Define the generating series 
\begin{align*}
	S_5(z)=&\sum_{m\geq 1}
h_{m-1}[(1-t^{-1})\bar{X}_{[k,n-1]}^{-1}+x_{n}+x_n^{-1}]z^{m-1}	\\
=& \frac{1}{(1-x_n z)(1-x_n^{-1}z)}
\prod_{i=k}^{n-1}\frac{1-t^{-1}x_i^{-1}z}{1-x_i^{-1}z}
\end{align*}
Then by direct computation we have
\begin{align*}
	tT_{n-1}^{-1}
	(S_5(z)\prod_{i=k}^{n-2}\frac{1-x_i^{-1}z}{1-t^{-1}x_i^{-1}z})
	=& tT_{n-1}^{-1}(\frac{1}{(1-x_n z)(1-x_n^{-1}z)}
	\frac{1-t^{-1}x_{n-1}^{-1}z}{1-x_{n-1}^{-1}z})\\
	=& -(1-t)\frac{1-t^{-1}x_{n-1}^{-1}z}
	{(1-x_{n-1}z)(1-x_{n-1}^{-1}z)(1-x_n z)(1-x_n^{-1}z)}\\
	&+\frac{1}{(1-x_{n-1}z)(1-x_{n-1}^{-1}z)}
\end{align*}
Note that 
\begin{align*}
	&T_k^{-1}...T_{n-2}^{-1}
	((\prod_{i=k}^{n-2}\frac{1-x_i^{-1}z}{1-t^{-1}x_i^{-1}z})
	\frac{1-t^{-1}x_{n-1}^{-1}z}
	{(1-x_{n-1}z)(1-x_{n-1}^{-1}z)})\\
	=& T_k^{-1}...T_{n-3}^{-1}((\prod_{i=k}^{n-3}
	\frac{1-x_i^{-1}z}{1-t^{-1}x_i^{-1}z})
	\frac{1-t^{-1}x_{n-2}^{-1}z}
	{(1-x_{n-2}z)(1-x_{n-2}^{-1}z)}
	\frac{(1-t^{-1}x_{n-1}^{-1}z)(1-t^{-1}x_{n-1}z)}
	{(1-x_{n-2}z)(1-x_{n-2}^{-1}z)})\\
	=& \frac{1-t^{-1}x_{k}^{-1}z}
	{(1-x_{k}z)(1-x_{k}^{-1}z)}
	\prod_{i=k+1}^{n-1}\frac{(1-t^{-1}x_{i}^{-1}z)(1-t^{-1}x_{i}z)}
	{(1-x_{i}z)(1-x_{i}^{-1}z)}
\end{align*}
Hence we have
\begin{align*}
	&t^{n-k}T_k^{-1}...T_{n-1}^{-1}
	h_{m-1}[(1-t^{-1})\bar{X}_{[k,n-1]}^{-1}+x_{n}+x_n^{-1}]\\
	=& h_{m-1}[x_k+x_k^{-1}]-
	\sum_{i=k}^{n-1}(1-t)t^{i-k}
	h_{m-1}[t^{-1}x_k+(1-t^{-1})\bar{X}_{[k,i]}^{-1}+x_{i+1}+x_{i+1}^{-1}]
\end{align*}
Again, by the telescoping argument we see
\begin{align*}
&\sum_{i=k}^{n-1}\frac{t^{i-k}}{(1-x_{i+1}z)(1-x_{i+1}^{-1}z)}
\prod_{j=k+1}^{i}\frac{(1-t^{-1}x_{j}^{-1}z)(1-t^{-1}x_{j}z)}
{(1-x_{j}z)(1-x_{j}^{-1}z)}\\
=& \frac{1}{(1-t^{-1}z^2)(1-t)}
\sum_{i=k}^{n-1}(\prod_{j=k+1}^{i}t^{i-k}
\frac{(1-t^{-1}x_{j}^{-1}z)(1-t^{-1}x_{j}z)}{(1-x_{j}z)(1-x_{j}^{-1}z)}\\
&-\prod_{j=k+1}^{i+1}t^{i-k+1}\frac{(1-t^{-1}x_{j}^{-1}z)(1-t^{-1}x_{j}z)}
{(1-x_{j}z)(1-x_{j}^{-1}z)})\\
=& \frac{1}{(1-t^{-1}z^2)(1-t)}(1
-\prod_{j=k+1}^{n}t^{n-k}\frac{(1-t^{-1}x_{j}^{-1}z)(1-t^{-1}x_{j}z)}
{(1-x_{j}z)(1-x_{j}^{-1}z)})
\end{align*}
Therefore this yields 
\begin{align*}
	&t^{n-k}T_k^{-1}...T_{n-1}^{-1}
	h_{m-1}[(1-t^{-1})\bar{X}_{[k,n-1]}^{-1}+x_{n}+x_n^{-1}]\\
	=& h_{m-1}[x_k+x_k^{-1}]-
	\sum_{i=0}^{\lfloor (m-1)/2 \rfloor}
	(t^{-i}h_{m-2i-1}[x_k+(1-t^{-1})x_{k}^{-1}]\\
	&-t^{n-k-i}h_{m-2i-1}[x_k+(1-t^{-1})x_k^{-1}
	+(1-t^{-1})\bar{X}_{[k+1,n]}^{-1}])
\end{align*}
\end{proof}

We can now show the stability of the action of $t^n Y_i^{(n)}$ on 
negative monomials. More specifically, we have
\begin{prop}\label{prop: stable limit of positive Y action 2}
	Let $x^{\mu}$ be a monomial with $\mu\in \mathbb{Z}^k_{\leq 0}$ for some 
	$k\geq 0$. Then for each fixed $i$, the sequence of Laurent polynomials 
	$$\{t^n Y_i^{(n)}x^{\mu}\}_{n\geq \max\{i,k\}}$$
	is convergent to 0.
\end{prop}
\begin{proof}
	For $i>k$ the sequence clearly converges to 0. 
	So we assume $i\leq k$, we have 
	$$t^n Y_i^{(n)}x^{\mu}=(t_0 t_{\infty})t^{2n-k}
	T_{i-1}...T_1 T_0^{-1}T_1^{-1}...(T_n^{(n)})^{-1}T_{n-1}^{-1}...T_k^{-1}
	(t^{k-i}T_{k-1}^{-1}...T_i^{-1}x^{\mu})$$
	Now that 
	$$(t^{k-i}T_{k-1}^{-1}...T_i^{-1}x^{\mu})$$
	is a polynomial in $x_1^{-1},...,x_k^{-1}$ and 
	can be written as a linear combination of monomials as
	$$(t^{k-i}T_{k-1}^{-1}...T_i^{-1}x^{\mu})=
	\sum_{\nu\in\mathbb{Z}^{k-1}_{\leq 0},m\geq 0}
	c_{\nu,m}(t)x^{\nu}x_k^{-m}$$
	where each $c_{\nu,m}(t)$ is a polynomial in $t$ with integer coefficients. 
	Hence it suffices to show the convergence of each monomial 
	$x^{\nu}x_k^{-m}$. Using the same argument as in 
	the proof of Proposition \ref{prop: stable limit of positive Y action 1} 
	we can assume $m>0$. Then by Corollary \ref{cor: raising part}, 
	for each monomial $x^{\nu}x_k^{-m}$ we have 
	\begin{align*}
		 (T_n^{(n)})^{-1}T_{n-1}^{-1}T_{k}^{-1}x^{\nu}x_k^{-m}
		=& x^{\nu}(h_m[(1-t^{-1})\bar{X}_{[k,n-1]}^{-1}+(1-t_n^{-1})x_{n}^{-1}+x_n] \\
		&- t_n^{-1}\beta h_{m-1}[(1-t^{-1})\bar{X}_{[k,n-1]}^{-1}+x_{n}+x_n^{-1}])
	\end{align*}
	By Corollary \ref{cor: symmetric part 3} we see 
	that 
	$$t^{2n-k}T_k^{-1}...T_{n-1}^{-1}
	(x^{\nu}h_m[(1-t^{-1})\bar{X}_{[k,n-1]}^{-1}+(1-t_n^{-1})x_{n}^{-1}+x_n])$$
	converges to 0 and by Lemma \ref{lem: symmetric part 4} we see
	\begin{align*}
		&t^{2n-k}T_k^{-1}...T_{n-1}^{-1}(x^{\nu}h_{m-1}[(1-t^{-1})\bar{X}_{[k,n-1]}^{-1}+x_{n}+x_n^{-1}])\\
		=& t^n(t^{n-k}T_k^{-1}...T_{n-1}^{-1}(x^{\nu}h_{m-1}[(1-t^{-1})\bar{X}_{[k,n-1]}^{-1}+x_{n}+x_n^{-1}]))
	\end{align*}
	also converges to 0. This finished the proof.
\end{proof}

Combining Proposition \ref{prop: stable limit of positive Y action 1} and 
Proposition \ref{prop: stable limit of positive Y action 2} we obtain the 
following stability result for the action of $t^n Y_i^{(n)}$.
\begin{cor}\label{cor: stability of positive Y action}
	Let $\mu\in\mathbb{Z}^k$ for some $k\geq 0$ and $\lambda$ be 
	a partition of length $k$. Then 
	the sequence of Laurent polynomials 
	$$\{t^n Y_i^{(n)}(x^{\mu}m_{\lambda}[\bar{X}_n^{\pm}])\}_{n\geq \max\{i,k\}}$$
	is convergent to an almost symmetric Laurent polynomial in $\Pas^{\pm}$.
\end{cor}
\begin{proof}
	We simply observe that 
	$$(t^{k-i}T_{k-1}^{-1}...T_i^{-1}x^{\mu})$$
	is a Laurent polynomial in $x_1^{\pm},...,x_k^{\pm}$ and 
	can be written as a linear combination of monomials as
	$$(t^{k-i}T_{k-1}^{-1}...T_i^{-1}x^{\mu})=
	\sum_{\nu\in\mathbb{Z}^{k-1},m\in \mathbb{Z}}
	c_{\nu,m} x^{\nu}x_k^{m}$$
	Then we have 
	By Proposition \ref{prop: stable limit of positive Y action 1} 
	and Proposition \ref{prop: stable limit of positive Y action 2} 
	we obtain the convergence for each monomial $x^{\nu}x_k^{m}$. Therefore 
	\begin{align*}
		&t^n Y_i^{(n)}(x^{\mu}m_{\lambda}[\bar{X}_n^{\pm}])\\
		=& t_0 	T_{i-1}...T_1 T_0^{-1}m_{\lambda}[\bar{X}_n^{\pm}]\\
		&T_1^{-1}...T_{k-1}^{-1}\sum_{\nu\in\mathbb{Z}^{k-1},m\in \mathbb{Z}}
		c_{\nu,m} (T_{k}^{-1}...T_{n}(T_n^{(n)})^{-1}T_{n-1}^{-1}...T_k^{-1}
		x^{\nu}x_k^{m})
	\end{align*}
	also converges in $\Pas^{\pm}$. Thus we obtain the 
	convergence for the whole expression\\
	$t^n Y_i^{(n)}(x^{\mu}m_{\lambda}[\bar{X}_n^{\pm}])$.
\end{proof}
\begin{ex}
	Let us consider the action of $t^n Y_1^{(n)}$ on the monomial 
	$x_1^{2}$. We have 
	\begin{align*}
		t^n Y_1^{(n)}x_1^{2}=& 
		q^2 t x_1^2 + qt(1-t)x_1(x_2+x_2^{-1}+...+x_n+x_n^{-1})
		+ (qt(\alpha+\beta)-t^n(1-t_0)\beta)x_1\\
		& + t(1-t)\alpha (x_2+x_2^{-1}+...+x_n+x_n^{-1}) 
		+ (t\alpha+q^{-1}t^n t_0\beta)x_1^{-1}\\
		& + (qt(1-t_0)+(1-t+t^n)t t_0+\alpha\beta (t-t^n)-t^n t_0 t_{\infty})
	\end{align*}
	The stable limit as $n\rightarrow \infty$ is 
	\begin{align*}
		\lim_{n\rightarrow \infty}t^n Y_1^{(n)}x_1^{2}=& 
		q^2 t x_1^2 + qt(1-t)x_1 e_1[X_1^{\pm}] + qt(\alpha+\beta)x_1\\
		& + \alpha t(1-t) e_1[X_1^{\pm}] + \alpha t x_1^{-1} + 
		(qt(1-t_0)+(1-t)t t_0+\alpha\beta t)
	\end{align*}
\end{ex}

\subsection{Positive stable limit DAHA action}
As we are done with the technical part proving the stability 
of the rescaled Cherednik operators in the previous section, 
we can now define the stable limit DAHA and its action on the almost 
symmetric polynomial space $\Pas^{\pm}$ 
as follows.

\begin{dfn}\label{def: Limit DAHA}
	The stable limit double affine Hecke algebra 
	$\mathscr{H}$ of type $(C^{\vee},C)$ 
	is the 
	$\mathbb{Q}(q,t,t_0,t_{\infty},a,c)$-algebra generated by 
  $$T_0,T_1,T_2,\dots,X_1^{\pm1},X_2^{\pm1},\dots,\tilde{Y}_1,\tilde{Y}_2,\dots$$
  satisfying the following relations:
	  \begin{subequations}\label{Limit DAHA relations}
		  \begin{equation}\label{Limit quadratic}
			  \begin{gathered}
			  (T_{0}-1)(T_{0}+t_0)=0, \\
			  (T_{i}-1)(T_{i}+t)=0, \quad i\geq 1,
			  \end{gathered}
			\end{equation}
		  \begin{equation}\label{Limit T relation ii}
			\begin{gathered}
			T_{i}T_{j}=T_{j}T_{i}, \quad |i-j|>1,\\
			T_{i}T_{i+1}T_{i}=T_{i+1}T_{i}T_{i+1}, \quad i\geq 1,\\
			T_0 T_1 T_0 T_1 = T_1 T_0 T_1 T_0,
			\end{gathered}
			\end{equation}
		  \begin{equation}\label{Limit X relation ii}
			  \begin{gathered}
				  t T_i^{-1} X_i T_i^{-1}=X_{i+1}, \quad i\geq 1\\
				  T_{i}X_{j}=X_{j}T_{i},\quad  j\neq i,i+1,\\
				  X_i X_j=X_j X_i \quad i,j\geq 1,\\
				  (X_1 T_0 - t_0 c^{-1})(X_1 T_0 + q^{-1}c)=0,
			  \end{gathered}
		  \end{equation}
		  \begin{equation}\label{Limit Y relation ii}
			  \begin{gathered}
				  t T_i \tilde{Y}_i T_i=\tilde{Y}_{i+1}, \quad i\geq 1\\
				  T_{i}\tilde{Y}_{j}=\tilde{Y}_{j}T_{i},\quad  j\neq i,i+1,\\
				  \tilde{Y}_i \tilde{Y}_j=\tilde{Y}_j \tilde{Y}_i \quad i,j\geq 1,\\
				  \tilde{Y}_1 X_1 T_0 \tilde{Y}_1 =0.
			  \end{gathered}
		  \end{equation}
	  \end{subequations}
  \end{dfn}
\begin{rmk}
	The last relation in \eqref{Limit Y relation ii} 
	can be regarded as the limit relation of the relation 
	$$(X_1T_0 (t^n Y_1^{(n)})  - t^n t_0 a)(X_1T_0 (t^n Y_1^{(n)})+ t^n t_0 t_{n} a^{-1})=0.$$
	in the finite DAHA.
\end{rmk}
Then from Corollary \ref{cor: stability of positive Y action} 
we obtain 
  \begin{thm}\label{def: positive Limit DAHA action}
	The stable limit double affine Hecke algebra 
	$\mathscr{H}$ of type $(C^{\vee},C)$ acts on the space $\Pas^{\pm}$ 
	of almost symmetric Laurent polynomials, with 
	the action of $Y_i$ defined as 
	$$\tilde{Y}_i = \lim_{n\rightarrow\infty}t^n Y_i^{(n)}\Pi_n$$
	and the action of $T_i,X_i^{\pm}$ defined in 
	Definition \ref{def: standard representation} for $n>i$.
  \end{thm}
  \begin{proof}
	From Corollary \ref{cor: stability of positive Y action} 
	we see the action of $\tilde{Y}_i$ is well-defined. And 
	the stability of the action of $T_i$ and $X_i$ are clear from the definition. 
	All relations in Definition \ref{def: Limit DAHA} 
	can be obtained by applying Proposition \ref{prop: continuity} 
	to the corresponding relations in the finite DAHA.
  \end{proof}

\subsection{Stabilization of the negative DAHA action}
In this section we repeat the stabilization procedure 
in the opposite side, namely  
the standard representation of DAHA on the Laurent polynomial ring with 
the Cherednik operators rescaled by $t^{-n}$, 
and use the $t^{-1}$-adic topology. 
The structure of the proof is parallel to the positive case, 
but the roles of the two plethystic contributions are reversed: 
terms that survived in the $t$-adic limit may now vanish, 
while terms involving negative powers are controlled in the $t^{-1}$-adic 
topology. We again start from the two separate summands in Corollary 
\ref{cor: raising part} and show the stability of each summand separately.
\begin{prop}\label{prop: symmetric part 5}
	For $k\in \mathbb{Z}_+$, let the sequence $\{a_n\}_{n\geq k}$ be defined by 
	\begin{align*}
		a_n =& t^{-n-k}T_{k}^{-1}...T_{n-1}^{-1}\\
		&(h_m[(1-t)\bar{X}_{[k,n-1]}+(1-t_{\infty})x_{n}+x_n^{-1}] 
		- h_m[(1-t)\bar{X}_{[k,n-1]}+(1-t_{\infty})x_{n}])
	\end{align*}
	where $t_{\infty}$ is a fixed parameter. 
	Then $a_n$ converges to 0 as $n\rightarrow \infty$.
\end{prop}
\begin{proof}
	Note that 
	$$T_i^{-1}= s_i + (1-t^{-1})x_{i+1}\frac{1-s_i}{x_{i+1}-x_i}$$
	Hence the action of $T_k^{-1}...T_{n-1}^{-1}$ on Laurent polynomials 
	only creates coefficients formed by negative powers of $t$. Since $m$ 
	is fixed, we see the coefficient of each monomial in the expression of 
	$$h_m[(1-t)\bar{X}_{[k,n-1]}+(1-t_{\infty})x_{n}+x_n^{-1}]
	- h_m[(1-t)\bar{X}_{[k,n-1]}+(1-t_{\infty})x_{n}]$$
	has the coefficient with degree at most $m$ in $t$. Therefore 
	the coefficient of each monomial in the expression of 
	$$t^{-k}T_k^{-1}...T_{n-1}^{-1}(
	h_m[(1-t)\bar{X}_{[k,n-1]}+(1-t_{\infty})x_{n}+x_n^{-1}]
	- h_m[(1-t)\bar{X}_{[k,n-1]}+(1-t_{\infty})x_{n}])$$
	has the coefficient with degree at most $m$ in $t$. Thus 
	$a_n$ converges to 0 as $n\rightarrow \infty$ in 
	the $t^{-1}$-adic topology.
\end{proof}
\begin{prop}\label{prop: symmetric part 6}
	For $k>0$ and $m\geq 0$, 
	let the sequence $\{a_n\}_{n\geq k}$ be defined by 
	$$a_n=t^{-k}T_k^{-1}...T_{n-1}^{-1}
	h_m[(1-t^{-1})\bar{X}_{[k,n-1]}^{-1}+(1-t_{\infty}^{-1})x_{n}^{-1}+x_n]$$
	where $t_{\infty}$ is a fixed parameter. Then the sequence 
	$\{a_n\}_{n\geq k}$ converges in $\Pas^{\pm}$ as 
	$n\rightarrow \infty$.
\end{prop}
\begin{proof}
From Corollary \ref{cor: symmetric part 3} we see that each $a_n$ is a symmetric 
Laurent polynomial in $x_{i},x_{i}^{-1}$ for all $k+1\leq i\leq n$, so we 
can write $a_n$ as a linear combination of almost symmetric polynomials 
$$a_n = \sum_{r,\lambda}c_{r,\lambda} x_{k}^{-r}
m_{\lambda}[\bar{X}_{[k,n]}^{\pm}],$$
summed over all $|r|\leq m$ and $|\lambda|\leq m$. It 
suffices to prove the stability of each coefficient $c_{r,\lambda}$ 
in front of the monomial $x_{k}^{r}x^{(0^{k-1},\lambda)}$ when 
$n$ is large enough. 
Now following the computation in the proof of Lemma \ref{lem: symmetric part 3} 
Corollary \ref{cor: symmetric part 3}, we still use the generating series 
\begin{align*}
	S_4(z)=&\sum_{m\geq 0}
h_m[(1-t^{-1})\bar{X}_{[k,n-1]}^{-1}+(1-t_{\infty}^{-1})x_{n}^{-1}+x_n]z^m	\\
=& \frac{1-t_{\infty}^{-1}x_{n}^{-1}z}{(1-x_n z)(1-x_n^{-1}z)}
\prod_{i=k}^{n-1}\frac{1-t^{-1}x_i^{-1}z}{1-x_i^{-1}z}
\end{align*}
Then from Lemma \ref{lem: symmetric part 3} we have
\begin{align*}
	T_{n-1}^{-1}
	(S_4(z)\prod_{i=k}^{n-2}\frac{1-x_i^{-1}z}{1-t^{-1}x_i^{-1}z})
	=& \frac{t^{-1}t_{\infty}(1-t_{\infty}^{-1}x_{n-1}^{-1}z)(1-t_{\infty}^{-1}x_n z)(1-t_{\infty}^{-1}x_n^{-1}z)}
	{(1-x_{n-1}z)(1-x_{n-1}^{-1}z)(1-x_n z)(1-x_n^{-1}z)}\\
	&+\frac{(1-t^{-1}t_{\infty}^{-1}x_{n-1}^{-1}z)((1-t^{-1}t_{\infty})+(t^{-1}-t_{\infty}^{-1})z^2)}
	{(1-x_{n-1}z)(1-x_{n-1}^{-1}z)(1-x_n z)(1-x_n^{-1}z)}
\end{align*}
Then by a recursion we see that
\begin{align*}
a_n=& t^{-k}\sum_{\substack{{\epsilon=\{\epsilon_0,...,\epsilon_{n-k-1}\}}
\\{\epsilon_i\in \{0,1\}}}}
\prod_{j=0}^{n-k-1}(\frac{(1-\bar{\epsilon}_j t_{\infty}^{-1}
	t^{-\sum_{l=0}^{j-1}\epsilon_l}x_{n-j}z)
(1-\bar{\epsilon}_j t_{\infty}^{-1} t^{-\sum_{l=0}^{j-1}\epsilon_l}x_{n-j}^{-1}z)}
{(1-x_{n-j} z)(1-x_{n-j}^{-1}z)}\\
&\prod_{j=0}^{n-k-1}(t^{-1+\sum_{l=0}^{j-1}\epsilon_l}t_{\infty})^{\epsilon_j}
((1-t^{-1+\sum_{l=0}^{j-1}\epsilon_l}t_{\infty})
(1-t_{\infty}^{-1} t^{-\sum_{l=0}^{j-1}\epsilon_l} z^2))^{\bar{\epsilon}_j})
\end{align*}
where $\bar{\epsilon}_j=1-\epsilon_j$. By exactly the same argument as in the proof of 
Corollary \ref{cor: symmetric part 1}, we obtain that 
the coefficient of $x_{k}^{r}x^{(0^{k-1},\lambda)}$ 
in $a_n$ is convergent as $n\rightarrow \infty$. 
This finished the proof.
\end{proof}
From the two propositions \ref{prop: symmetric part 5} and 
\ref{prop: symmetric part 6} and by the same argument as in 
Corollary \ref{cor: stability of positive Y action} 
we obtain a similar stability result for the action of 
$t^{-n}Y_i^{(n)}$ on $\Pas^{\pm}$.
\begin{prop}\label{prop: stability of negative Y action}
	Let $\mu\in\mathbb{Z}^k$ for some $k\geq 0$ and $\lambda$ be 
	a partition of length $k$. Then 
	the sequence of Laurent polynomials 
	$$\{t^{-n} Y_i^{(n)}(x^{\mu}m_{\lambda}[\bar{X}_n^{\pm}])\}_{n\geq \max\{i,k\}}$$
	is convergent to an almost symmetric Laurent polynomial in $\Pas^{\pm}$.	
\end{prop}
Finally, we can define the negative stable limit DAHA and 
its action on $\Pas^{\pm}$ as follows. 
\begin{thm}\label{def: negative Limit DAHA action}
	The stable limit double affine Hecke algebra 
	$\mathscr{H}$ of type $(C^{\vee},C)$ acts on the space $\Pas^{\pm}$ 
	of almost symmetric Laurent polynomials, with 
	the action of $\tilde{Y}_i$ defined as 
	$$\tilde{Y}_i = \lim_{n\rightarrow\infty}t^{-n} Y_i^{(n)}\Pi_n$$
	and the action of $T_i,X_i^{\pm}$ defined in 
	Definition \ref{def: standard representation} for $n>i$.
\end{thm}
\begin{rmk}
	We are using the same stable limit DAHA $\mathscr{H}$ 
	for both the positive and negative stabilization, since 
	the last relation in \eqref{Limit Y relation ii} 
	is also the limit relation in the $t^{-1}$-adic sense of the relation 
	$$(X_1T_0 (t^{-n} Y_1^{(n)})  - t^{-n} t_0 a)(X_1T_0 (t^{-n} Y_1^{(n)})+ t^{-n} t_0 t_{n} a^{-1})=0.$$
	in the finite DAHA. We will then be 
	using $\tilde{Y}_i^+$ and $\tilde{Y}_i^-$ to 
	denote and distinguish the two different actions of
	the positive and negative stable limit DAHA actions.
\end{rmk}
\begin{ex}
	Let us consider the action of $t^{-n} Y_1^{(n)}$ on the monomial 
	$x_1^{-1}$. We have 
	\begin{align*}
		t^{-n} Y_1^{(n)}x_1^{-1}=& 
		q^{-1} t^{-1} t_0 t_{\infty} x_1^{-1} 
		- t^{-1}(1-t^{-1})t_0 t_{\infty}(x_2+x_2^{-1}+...+x_n+x_n^{-1})\\
		& + (q t^{-1} (1-t^{-1}) t_{\infty} 
		-t^{-1}t_{\infty}(1-t_0)-qt^{-n}(1-t^{-1}t_{\infty}))x_1\\
		&+ (t^{-1} (1-t^{-1}) t_{\infty}\alpha
		-t^{-1}(t_0\beta+q^{-1}t_{\infty}\alpha)-\alpha t^{-n}(1-t^{-1}t_{\infty}))
	\end{align*}
	The stable limit as $n\rightarrow\infty$ is 
	\begin{align*}
		\lim_{n\rightarrow\infty}t^{-n} Y_1^{(n)}x_1^{-1}=& 
		q^{-1} t^{-1} t_0 t_{\infty} x_1^{-1} 
		- t^{-1}(1-t^{-1})t_0 t_{\infty}e_1[X_1^{\pm}]\\
		& + (q t^{-1} (1-t^{-1}) t_{\infty} 
		-t^{-1}t_{\infty}(1-t_0))x_1\\
		&+ (t^{-1} (1-t^{-1}) t_{\infty}\alpha
		-t^{-1}(t_0\beta+q^{-1}t_{\infty}\alpha))
	\end{align*}
\end{ex}

The results of this section complete 
the construction of the stable limit DAHA action of type $(C^\vee,C)$ 
on almost symmetric Laurent polynomials. 
This provides the algebraic foundation 
for studying stable limit Koornwinder polynomials and for 
comparing the $(C^\vee,C)$ theory with the type $A$ stable limit DAHA.

\section{Stable limit Koornwinder polynomials}\label{sec: stable limit Koornwinder polynomials}
In this section we will study the stable limit Koornwinder polynomials. 
As shown in \cite{BW,IW2}, the non-symmetric Macdonald polynomials of type $A$ 
do stabilize in infinite rank with respect to the $t$-adic topology and 
still serve as the eigenfunctions of the limit Cherednik operators. 
But they do not form a basis of the almost symmetric polynomial space 
which is too large. A tail symmetrization process 
is needed to obtain sufficient eigenfunctions of the limit Cherednik operators. 
In type $(C^\vee,C)$, the situation is similar. Therefore, 
we will break the stabilization into two steps. 
First we will show that the non-symmetric Koornwinder polynomials 
stabilize in infinite rank. Second we will use the tail symmetrization 
process to obtain all other stable limit Koornwinder polynomials 
which will form a eigenbasis of the limit Cherednik operators. 
We will discuss the stabilization for both the positive and negative stable limit DAHA actions.
\subsection{Triangularity}
The following result is a consequence of the triangularity property stated in 
\cite{Sto}*{Proposition 4.5}.
\begin{prop}\label{prop: triangularity of Y}
	Let $\lambda\in\mathbb{Z}^n$ and $1\leq i \leq n$. Then we have 
	$$Y_i^{(n)}x^{\lambda}\in q^{\lambda_i}(t_0t_n)^{\delta(\lambda_i\leq 0)} 
	t^{-\epsilon(\lambda_i)\beta_{\lambda}(i)}x^{\lambda}
	+\sum_{\mu<\lambda}\mathbb{K}'x^{\mu},$$
	where $\delta(\lambda_i\leq 0)=1$ if $\lambda_i\leq 0$ and 0 otherwise, 
	$\epsilon(\lambda_i)=1$ if $\lambda_i>0$ and $-1$ otherwise, and 
	\[
	\beta_{\lambda}(i)=
	\begin{cases}
	\#\{j<i|\lambda_j\preceq\lambda_i\}+\#\{j\geq i|\lambda_j\prec\lambda_i\}, & \text{if } \lambda_i>0,\\
	\#\{j\leq i|\lambda_j\prec\lambda_i\}+\#\{j>i|\lambda_j\preceq \lambda_i\}, & \text{if } \lambda_i\leq 0
	\end{cases}
	\]
	in which the ordering $\prec$ on integers is defined by
	$$0\prec 1\prec -1\prec 2\prec -2\prec ...$$
\end{prop}

We then first prove  
the triangularity property of the positive stable limit Cherednik operators 
on the monomial basis of $\Pas^{\pm}$. The proof is basically 
the same as the proof of the triangularity property of the stable limit 
Cherednik operators in type $A$ presented in \cite{IW2}.
\begin{prop}\label{prop: triangularity of positive limit Y}
	For symbol $\lambda\vert\mu$ and $i\geq 1$, we have 
	$$\tilde{Y}_i^+ m_{\lambda\vert\mu}\in 
	\delta_i(\lambda)q^{\lambda_i}
	t^{u_{\lambda\mu}(i)}m_{\lambda\vert\mu}
	+\sum_{\nu\vert\eta<\lambda\vert\mu}\mathbb{K}'m_{\nu\vert\eta},$$
	where $\delta_i(\lambda)=1$ if $i\leq l(\lambda)$ and 
	$\lambda_i>0$, and 0 otherwise, and 
	\[
	u_{\lambda}(i)=
	\begin{cases}
	\#\{j<i|\lambda_j\succ\lambda_i\}+
	\#\{j\geq i|\lambda_j\succeq\lambda_i\}, & \text{if } \lambda_i>0,\\
	\#\{j\leq i|\lambda_j\succeq\lambda_i\}+
	\#\{j> i|\lambda_j\succ\lambda_i\}, & \text{if } \lambda_i\leq 0
	\end{cases}	
	\]
\end{prop}
\begin{proof}
	Suppose that $l(\lambda)=k$. 
	We first assume $i\leq k$. Note first that 
	$$\tilde{Y}_i^+ m_{\lambda\vert\mu}\in \P(k)^{\pm}.$$
	This is because for each $n> k$, we have that 
	$\Pi_n m_{\lambda\vert\mu}$ is symmetric in $x_{k+1},...,x_n$, and 
	also $x_j$ and $x_j^{-1}$ for $k\leq j\leq n$. Therefore for 
	all $k\leq j\leq n$ we have 
	$$t^n Y_i^{(n)}\Pi_n m_{\lambda\vert\mu}
	=t^n Y_i^{(n)}T_j^{(n)}\Pi_n m_{\lambda\vert\mu}
	=T_j^{(n)}(t^n Y_i^{(n)}\Pi_n m_{\lambda\vert\mu}).$$
	By Remark \ref{rmk: symmetry} we see that 
	$t^n Y_i^{(n)}\Pi_n m_{\lambda\vert\mu}$ is also 
	symmetric in $x_{k+1},...,x_n$ and in $x_j$ and 
	$x_j^{-1}$ for $k\leq j\leq n$. By Theorem 
	\ref{def: positive Limit DAHA action} we can 
	take the limit to obtain that 
	$\tilde{Y}_i^+ m_{\lambda\vert\mu}\in \P(k)^{\pm}$. 
	
	Then by 
	Proposition \ref{prop: standard basis 1} $\tilde{Y}_i^+ m_{\lambda\vert\mu}$ 
	is a linear combination of basis elements of the form 
	$m_{\nu\vert\eta}$ with $l(\nu)\leq k$. Consider a basis 
	element $m_{\nu\vert\eta}$ that appears in the linear combination of 
	$\tilde{Y}_i^+ m_{\lambda\vert\mu}$ with non-zero coefficient. 
	Then the monomial $x^{\nu 0^{k-l(\nu)}\eta}$ must appear 
	in $Y_i^{(n)}\Pi_n m_{\lambda\vert\mu}$ with non-zero coefficient for 
	sufficiently large $n$. By the triangularity property 
	Proposition \ref{prop: triangularity of Y}of 
	$Y_i^{(n)}$ we see that $\nu 0^{k-l(\nu)}\eta\leq \lambda\mu$ 
	and therefore by Proposition \ref{prop: Bruhat order} we have 
	$\nu\vert\eta\leq \lambda\vert\mu$. 

	It remains to determine the top coefficient of $m_{\lambda\vert\mu}$. 
	For $n$ sufficiently large, the largest monomial 
	that appears in the expansion of $\Pi_n m_{\lambda\vert\mu}$ 
	is $x^{\lambda \mu^{-} 0^{n-k-l(\mu)}}$ 
	with coefficient 1. If $\lambda_i\leq 0$, 
	by Proposition \ref{prop: triangularity of Y} the coefficient of 
	$x^{\lambda \mu^{-} 0^{n-k-l(\mu)}}$ in the expansion of 
	$t^n Y_i^{(n)}x^{\lambda \mu^{-} 0^{n-k-l(\mu)}}$ is therefore
	$$q^{\lambda_i}(t_0t_n)t^{\beta_{\lambda\mu^{-}0^{n-k-l(\mu)}}(i)+n}
	=q^{\lambda_i}(t_0t_n)t^{\beta_{\lambda\mu^{-}}(i)+2n-k-l(\mu)}$$
	and converges to 0 as $n\rightarrow \infty$. If $\lambda_i>0$, 
	this coefficient is 
	$$q^{\lambda_i}t^{n-\beta_{\lambda\mu^{-}0^{n-k-l(\mu)}}(i)}
	=q^{\lambda_i}t^{u_{\lambda\mu^{-}}(i)}=q^{\lambda_i}t^{u_{\lambda\mu}(i)}$$
	which is exactly as stated. Note the last equality holds since 
	$\lambda_i>0$.

	For $i>k$, we observe that
	$$\pi_{n+1} t^{n+1}Y_i^{(n+1)}\Pi_{n+1}m_{\lambda\vert\mu}
	=t^2(t^{n}Y_i^{(n)}\Pi_{n}m_{\lambda\vert\mu})$$
	for $n>i$. Now by the same argument as above 
	$\tilde{Y}_i^+ m_{\lambda\vert\mu}\in \P(i)^{\pm}$. 
	Then by Proposition \ref{prop: standard basis 1} 
	$\tilde{Y}_i^+ m_{\lambda\vert\mu}$ is a linear combination of 
	basis elements of the form 
	$m_{\nu\vert\eta}$ with $l(\nu)\leq i$. When $n\geq i+|\lambda^+|+|\mu|$, 
	from the observation above we have 
	$$[m_{\nu\vert\eta}]t^{n+1}Y_i^{(n+1)}\Pi_{n+1}m_{\lambda\vert\mu}
	=t^2[m_{\nu\vert\eta}](t^{n}Y_i^{(n)}\Pi_{n}m_{\lambda\vert\mu})$$
	which converges to 0 as $n\rightarrow \infty$. Therefore 
	$\tilde{Y}_i^+ m_{\lambda\vert\mu}=0$ for $i>k$. This 
	completes the proof.
\end{proof}

Using exactly the same argument we can also prove the following 
triangularity property of the negative stable limit Cherednik 
operators on the monomial basis of $\Pas^{\pm}$.
\begin{prop}\label{prop: triangularity of negative limit Y}
	For symbol $\lambda\vert\mu$ and $i\geq 1$, we have 
	$$\tilde{Y}_i^- m_{\lambda\vert\mu}\in 
	\delta_i'(\lambda)q^{\lambda_i}(t_0t_{\infty})
	t^{-u_{\lambda\mu}(i)}m_{\lambda\vert\mu}
	+\sum_{\nu\vert\eta<\lambda\vert\mu}\mathbb{K}'m_{\nu\vert\eta},$$
	where $\delta_i'(\lambda)=1$ if $i\leq l(\lambda)$ and 
	$\lambda_i\leq 0$, and 0 otherwise.
\end{prop}
By setting $\mu=\emptyset$ in Propositions 
\ref{prop: triangularity of positive limit Y} and 
\ref{prop: triangularity of negative limit Y} we obtain the 
following corollary.
\begin{cor}\label{cor: limit of Y coefficients}
	Let $\lambda\in\mathbb{Z}^k$ for some $k\geq 0$ and 
	$i\geq 1$. Then we have the expansion formula
	$$Y_i^{(n)}x^{\lambda} = 
	\sum_{\substack{{\nu\leq \lambda 0^{n-k}}\\
	{\{\nu_{k+1},...,\nu_{l(\nu)}\}\in\Pi}}} c_{\lambda0^{n-k},\nu,i}^{(n)}
	\sum_{\sigma\in S_{1^k,n-k}\rtimes\mathbb{Z}^{n-k}}x^{\sigma\nu}$$
	for $n>\max\{k,i\}$, and the coefficients $c_{\lambda0^{n-k},\nu,i}^{(n)}$ 
	satisfy the following properties:
	\begin{enumerate}
		\item For each $\nu\leq \lambda0^{n-k}$, the sequence 
		$\{t^n c_{\lambda0^{n-k},\nu,i}^{(n)}\}$ has a $t$-adic limit in 
		$\mathbb{K}'$ as $n\rightarrow \infty$. 
		In particular, 
		$$ \lim_{n\rightarrow \infty} t^n c_{\lambda0^{n-k},\lambda0^{n-k},i}^{(n)} =
		\delta_i(\lambda)q^{\lambda_i}t^{u_{\lambda}(i)}$$
		Furthermore, if 
		$ \lim_{n\rightarrow \infty} t^n c_{\lambda0^{n-k},\nu,i}^{(n)}=0$, 
		then the sequence 
		$\{c_{\lambda0^{n-k},\nu,i}^{(n)}\}$ also converges $t$-adically  
		as $n\rightarrow \infty$.
		\item For each $\nu\leq \lambda0^{n-k}$, the sequence 
		$\{t^{-n} c_{\lambda0^{n-k},\nu,i}^{(n)}\}$ has a $t^{-1}$-adic limit in 
		$\mathbb{K}'$ as $n\rightarrow \infty$. 
		In particular, 
		$$ \lim_{n\rightarrow \infty} t^{-n} c_{\lambda0^{n-k},\lambda0^{n-k},i}^{(n)} =
		\delta_i'(\lambda)q^{\lambda_i}(t_0t_{\infty})
		t^{-u_{\lambda}(i)}$$
		Furthermore, if 
		$ \lim_{n\rightarrow \infty} t^{-n} c_{\lambda0^{n-k},\nu,i}^{(n)}=0$, 
		then the sequence 
		$\{c_{\lambda0^{n-k},\nu,i}^{(n)}\}$ also converges $t^{-1}$-adically  
		as $n\rightarrow \infty$.
	\end{enumerate}	
\end{cor}

\subsection{Stablization of non-symmetric Koornwinder polynomials}
In this section we will process the stabilization of the non-symmetric 
Koornwinder polynomials. 
Consider the double affine Hecke algebra of 
type $(C_n^{\vee},C_n)$ of finite rank $n$. 
The Koornwinder polynomials $E_{\lambda}\in\P_n^{\pm}$, 
$\lambda\in\mathbb{Z}^n$ 
are defined as the simultaneous eigenbasis of 
the Cherednik operators $Y_i^{(n)}$. Each 
$E_{\lambda}$ satisfies the triangularity property 
$$E_{\lambda}\in x^{\lambda}+\sum_{\mu<\lambda}\mathbb{K}'x^{\mu}$$
and has the eigenvalues explicitly given by
$$Y_i^{(n)}E_{\lambda}=q^{\lambda_i}(t_0t_n)^{\delta(\lambda_i\leq 0)}
t^{-\epsilon(\lambda_i)\beta_{\lambda}(i)}E_{\lambda}.$$

The Koornwinder polynomials satisfy the following recursion relation 
\begin{prop}[See \cite{Sa}*{Theorem 4.1}]\label{prop: recursion of Koornwinder}
	Suppose $\lambda\in\mathbb{Z}^n$. Then 
	\begin{enumerate}
		\item For $1\leq i\leq n-1$, if 
	$\lambda_i> \lambda_{i+1}$, we have
	$$E_{s_i\lambda}=
	(T_i+\frac{(1-t)q^{\lambda_{i+1}}(t_0t_n)^{\delta(\lambda_{i+1}\leq 0)}
	t^{-\epsilon(\lambda_{i+1})\beta_{\lambda}(i+1)}}
	{q^{\lambda_i}(t_0t_n)^{\delta(\lambda_i\leq 0)}
	t^{-\epsilon(\lambda_i)\beta_{\lambda}(i)}
	-q^{\lambda_{i+1}}(t_0t_n)^{\delta(\lambda_{i+1}\leq 0)}
	t^{-\epsilon(\lambda_{i+1})\beta_{\lambda}(i+1)}})E_{\lambda}$$
	and if $\lambda_i=\lambda_{i+1}$, we have 
	$E_{s_i\lambda}=T_i E_{\lambda}.$
	\item If $\lambda_n>0$, we have
	$$E_{s_n\lambda}= (T_n+\frac{(1-t_n)
	q^{-\lambda_n}(t_0t_n)
	t^{\beta_{\lambda}(n)}
	+(1-t_0)t_n}
	{q^{\lambda_n}t^{-\beta_{\lambda}(n)}-q^{-\lambda_n}(t_0t_n)
	t^{\beta_{\lambda}(n)}})E_{\lambda}$$
	and if $\lambda_n=0$, we have 
	$E_{s_n\lambda}=T_n E_{\lambda}.$
	\end{enumerate}
\end{prop}

\begin{thm}\label{thm: stable limit Koornwinder}
	Let $\lambda\in\mathbb{Z}^k$ with $\lambda_k\neq 0$. Then the 
	sequence of Koornwinder polynomials $\{E_{\lambda 0^n}\}_{n\geq 0}$ 
	is convergent $t$-adically 
	to an almost symmetric Laurent polynomial $\E_{\lambda}^+$ 
	and $t^{-1}$-adically to an almost symmetric Laurent polynomial 
	$\E_{\lambda}^-$ in $\Pas^{\pm}$ as $n\rightarrow \infty$. Furthermore, 
	$\E_{\lambda}^+$ and $\E_{\lambda}^-$ are respectively 
	eigenfunctions of the positive and negative stable limit Cherednik 
	operators $\tilde{Y}_i^+$ and $\tilde{Y}_i^-$ with eigenvalues
	$$\tilde{Y}_i^+\E_{\lambda}^+ = \delta_i(\lambda)q^{\lambda_i}
	t^{u_{\lambda}(i)}\E_{\lambda}^+$$
	and 
	$$\tilde{Y}_i^-\E_{\lambda}^- = \delta_i'(\lambda)q^{\lambda_i}(t_0t_{\infty})
	t^{-u_{\lambda}(i)}\E_{\lambda}^-.$$	
	We call $\E_{\lambda}^+$ the positive 
	stable limit Koornwinder polynomial 
	and $\E_{\lambda}^-$ the negative stable limit Koornwinder polynomial 
	associated to the symbol $\lambda$.
\end{thm}
\begin{proof}
	We just prove the positive case and the proof of the negative 
	case is exactly the same. 
	First note that each $E_{\lambda 0^n}$ is symmetric 
	in $x_{k+1},...,x_n$ and in $x_i$ and $x_i^{-1}$ for $k+1\leq i\leq n$ 
	by Proposition \ref{prop: recursion of Koornwinder}.  
	Then we can write
	the monomial expansion of $E_{\lambda 0^n}$ as 
	$$E_{\lambda 0^n} = \sum_{\substack{{\nu\leq \lambda 0^n}\\
	{\{\nu_{k+1},...,\nu_{k+n}\}\in \Pi}}} a_{\lambda 0^n,\nu}^{(n)} 
	\sum_{\sigma\in S_{1^k,n}\rtimes\mathbb{Z}^{n}}x^{\sigma\nu}.$$
	One important fact to be noted is 
	that for a fixed $\lambda$, $\nu$ has only finite many choices. 
	Now it suffices to prove the convergence of the coefficients 
	$a_{\lambda 0^n,\nu}^{(n)}$ for each $\nu\leq \lambda 0^n$ as 
	$n\rightarrow \infty$. We will prove this by induction 
	on the Bruhat order of $\nu$. Clearly 
	the top coefficient $a_{\lambda 0^n,\lambda 0^n}^{(n)}=1$ 
	and therefore is convergent. Suppose 
	that for all $\nu<\nu'$, the sequence 
	$\{a_{\lambda 0^n,\nu'}^{(n)}\}_{n\geq 0}$ is convergent. 
	Choose $i$ such that 
	$$c_{\lambda 0^n,\lambda 0^n,i}^{(n+k)}\neq c_{\nu\nu,i}^{(n+k)}.$$
	This is possible since the spectrum 
	of the Cherednik operators is simple. 
	By comparing the 
	coefficient of $x^{\nu}$ in the eigenvalue equation 
	$$Y_i^{(n+k)}E_{\lambda 0^n} = 
	c_{\lambda 0^n,\lambda 0^n,i}^{(n+k)} E_{\lambda 0^n}$$
	and applying Corollary \ref{cor: limit of Y coefficients}, we have 
	$$a_{\lambda 0^n,\nu}^{(n)}=
	\frac{\sum_{\nu<\nu'\leq \lambda 0^n}c_{\nu'\nu,i}^{(n+k)}
	a_{\lambda 0^n,\nu'}^{(n)}}
	{c_{\lambda 0^n,\lambda 0^n,i}^{(n+k)}-c_{\nu\nu,i}^{(n+k)}}=
	\frac{\sum_{\nu<\nu'\leq \lambda 0^n}(t^{n+k} c_{\nu'\nu,i}^{(n+k)})
	a_{\lambda 0^n,\nu'}^{(n)}}
	{(t^{n+k} c_{\lambda 0^n,\lambda 0^n,i}^{(n+k)})-(t^{n+k} c_{\nu\nu,i}^{(n+k)})}.$$
	If $(t^{n+k} c_{\lambda 0^n,\lambda 0^n,i}^{(n+k)})-(t^{n+k} c_{\nu\nu,i}^{(n+k)})$ 
	does not converge to zero, then 
	by the induction hypothesis and Corollary 
	\ref{cor: limit of Y coefficients} the sequence $a_{\lambda 0^n,\nu}^{(n)}$
	is convergent as $n\rightarrow \infty$. Otherwise, we have 
	$$\lim_{n\rightarrow\infty}
	(t^{n+k} c_{\lambda 0^n,\lambda 0^n,i}^{(n+k)})-
	(t^{n+k} c_{\nu\nu,i}^{(n+k)})=0.$$
	Since $c_{\lambda 0^n,\lambda 0^n,i}^{(n+k)}$ and 
	$c_{\nu\nu,i}^{(n+k)}$ are both monomials and not equal, both 
	$t^{n+k} c_{\lambda 0^n,\lambda 0^n,i}^{(n+k)}$ 
	and $t^{n+k} c_{\nu\nu,i}^{(n+k)}$ converge to 0 as $n\rightarrow \infty$. 
	This is equivalent to saying that $\lambda_i\leq 0$ and $\nu_i\leq 0$. 
	Then by the explicit formula of $c_{\lambda 0^n,\lambda 0^n,i}^{(n+k)}$ and 
	$c_{\nu\nu,i}^{(n+k)}$ in Proposition \ref{prop: triangularity of Y} 
	we have that they both converge $t$-adically to different monomials. 
	Furthermore, since $\nu_i\leq 0$, we have that
	$$\lim_{n\rightarrow \infty} t^{n+k} c_{\nu'\nu,i}^{(n+k)}=0.$$
	for all $\nu'>\nu$. By Corollary \ref{cor: limit of Y coefficients} this 
	implies that the sequence $\{c_{\nu'\nu,i}^{(n+k)}\}$ 
	converges $t$-adically as $n\rightarrow \infty$. 
	Then we use the first equation above for $a_{\lambda 0^n,\nu}^{(n)}$ to see that 
	$\{a_{\lambda 0^n,\nu}^{(n)}\}$ is convergent as $n\rightarrow \infty$. 
	Then by induction we obtain the convergence of all coefficients 
	and therefore the convergence of the sequence 
	$\{E_{\lambda 0^n}\}_{n\geq 0}$. The 
	eigenvalue property of $\E_{\lambda}^+$ follows from the 
	application of Proposition \ref{prop: continuity} 
	on the sequence of operators $\{t^{n+k} Y_{i}^{(n+k)}\}$ acting on 
	the sequence $\{E_{\lambda 0^n}\}$.
\end{proof}
\begin{ex}
	\begin{enumerate}
		\item When $\lambda=(0,1)$, we have 
		\begin{align*}
			E_{(0,1,0^n)}=&x_2
			+\frac{q(1-t)}{q-t^{2n+1}t_0 t_n}x_1
			+\frac{\alpha+t^{n+1} t_0\beta}{q-t^{2n+1}t_0 t_n}\\
			=&x_2+\frac{qt^{-2n-1}(1-t)}{qt^{-2n-1}-t_0 t_n}x_1
			+\frac{t^{-n}(t^{-n-1}\alpha+t_0\beta)}{t^{-2n-1}q-t_0 t_n}
		\end{align*}
		and therefore 
		$$\E_{(0,1)}^+=x_2+(1-t)x_1+q^{-1}\alpha,\qquad \E_{(0,1)}^-=x_2.$$
		The action of Cherednik operators on the stable limit 
		Koornwinder polynomials is
		$$Y_1^+\E_{(0,1)}^+=0,\qquad Y_2^+\E_{(0,1)}^+=qt\E_{(0,1)}^+$$
		and 
		$$Y_1^-\E_{(0,1)}^-=t^{-2}t_0t_{\infty}\E_{(0,1)}^-,\qquad 
		Y_2^-\E_{(0,1)}^-=0.$$
		\item When $\lambda=(-1)$, we have 
		\begin{align*}
		E_{(-1,0^n)}=&x_1^{-1}+
		\frac{q(t-1)}{t-q}(x_2+x_2^{-1}+...+x_n+x_n^{-1})\\
		+&\frac{q^3(t-1)-(q-1)q^2 t^n(t-t_n)-qt^{2n}t_n(qt+tt_0-t-qt_0)}
		{(t-q)(q^2-t^{2n}t_0t_n)}x_1\\
		+&\frac{-(q^2 \alpha +q^3 \beta)
		+q t^n(t\alpha-t_n\alpha+qt\beta-qt_0\beta)
		+t^{2n+1}(t_n\alpha+qt_0\beta)}
		{(t-q)(q^2-t^{2n}t_0t_n)}
		\end{align*}
		and therefore
		$$\E_{(-1)}^+=x_1^{-1}+\frac{q(t-1)}{t-q}e_1[X_1^{\pm}]
		+\frac{q(t-1)}{t-q}x_1-\frac{\alpha +q \beta}{t-q}$$
		and 
		$$\E_{(-1)}^-=x_1^{-1}+\frac{q(t-1)}{t-q}e_1[X_1^{\pm}]
		+\frac{q(qt+tt_0-t-qt_0)}{(t-q)t_0}x_1-
		\frac{t(t_{\infty}\alpha+qt_0\beta)}{(t-q)t_0t_{\infty}}.$$
	\end{enumerate}
\end{ex}

As a corollary, we have
\begin{cor}\label{cor: expansion of stable limit Koornwinder}
Let $\lambda\in\mathbb{Z}^k$ with $\lambda_k\neq 0$ and $N=|\lambda^+|$. 
Then we have an expansion formula
$$E_{\lambda 0^{N}} = \sum_{\substack{{\nu\eta\leq \lambda 0^{N}}\\
{\nu_{l(\nu)}\neq 0,\eta\in\Pi}}} 
\tilde{a}_{\lambda 0^{N},\nu\eta}^{(N)}
x^{\nu}m_{\eta}[\bar{X}_{[l(\nu)+1,k+N]}^{\pm}]$$
with all summands $\nu\eta$ satisfying $l(\nu)\leq k$. 
Furthermore, for all $n>N$, 
$$E_{\lambda 0^{n}} = \sum_{\substack{{\nu\eta\leq \lambda 0^{N}}\\
{\nu_{l(\nu)}\neq 0,\eta\in\Pi}}} 
\tilde{a}_{\lambda 0^{N},\nu\eta}^{(n)}
x^{\nu}m_{\eta}[\bar{X}_{[l(\nu)+1,k+n]}^{\pm}].$$
for some $\tilde{a}_{\lambda 0^{N},\nu\eta}^{(n)}\in \mathbb{K}'$. 
In particular, 
$\tilde{a}_{\lambda 0^{N},\lambda 0^{N}}^{(n)}=1$ for all $n>N$. 
The sequence $\{\tilde{a}_{\lambda 0^{N},\nu\eta}^{(n)}\}_{n>N}$ 
is convergent $t$-adically (resp. $t^{-1}$-adically) to some 
$\tilde{a}_{\lambda,\nu\eta}^+\in \mathbb{K}'$ 
(resp. $\tilde{a}_{\lambda,\nu\eta}^-\in \mathbb{K}'$)
as $n\rightarrow \infty$. 
Finally, we have the expansion formulae
$$\E_{\lambda}^+ = \sum_{\substack{{\nu\eta\leq \lambda 0^{N}}\\
{\nu_{l(\nu)}\neq 0,\eta\in\Pi}}} 
\tilde{a}_{\lambda 0^{N},\nu\eta}^+
m_{\nu\vert\eta},\qquad
\E_{\lambda}^- = \sum_{\substack{{\nu\eta\leq \lambda 0^{N}}\\
{\nu_{l(\nu)}\neq 0,\eta\in\Pi}}}
\tilde{a}_{\lambda 0^{N},\nu\eta}^-
m_{\nu\vert\eta}$$
with 
$\tilde{a}_{\lambda 0^{N},\lambda 0^{N}}^+ =
\tilde{a}_{\lambda 0^{N},\lambda 0^{N}}^- = 1.$
\end{cor}
\begin{proof}
	From the proof of Theorem \ref{thm: stable limit Koornwinder} we have 
	\begin{align*}
		E_{\lambda 0^N} =& \sum_{\substack{{\nu\leq \lambda 0^N}\\
		{\{\nu_{k+1},...,\nu_{k+N}\}\in \Pi}}} a_{\lambda 0^N,\nu}^{(N)} 
		\sum_{\sigma\in S_{1^k,N}\rtimes\mathbb{Z}^{N}}x^{\sigma\nu}\\
		=& \sum_{\substack{{\nu\leq \lambda 0^N}\\
		{\{\nu_{k+1},...,\nu_{k+N}\}\in \Pi}}} a_{\lambda 0^N,\nu}^{(N)} 
		\sum_{\sigma\in S_{1^k,N}\rtimes\mathbb{Z}^{N}}x^{\nu^{L}}
		m_{\nu^{R}}[\bar{X}_{[k+1,k+N]}]
	\end{align*}	
	where $\nu^{L}=(\nu_1,...,\nu_k)$ and $\nu^{R}=(\nu_{k+1},...,\nu_{k+N})$. 
	If $\nu^{L}_{k}\neq 0$ or $\nu^{R}=\emptyset$, 
	then 
	$$x^{\nu^{L}}m_{\nu^{R}}[\bar{X}_{[k+1,k+N]}]=
	x^{\nu^{L,\circ}}m_{\nu^{R}}[\bar{X}_{[l(\nu^{L,\circ})+1,k+N]}^{\pm}]$$
	where 
	$\nu^{L,\circ}=(\nu_1,...,\nu_{l(\nu^{L,\circ})})$ is the subpartition 
	of $\nu^{L}$ satisfying $\nu_{l(\nu^{L,\circ})}\neq 0$, 
	$\nu_{l(\nu^{L,\circ})+1}=\cdots=\nu_k=0$, 
	and hence $l(\nu^{L,\circ})\leq k$. Otherwise, we use the formula in Example 
	\ref{ex: monomial expansion} and the method in the proof of Proposition 
	\ref{prop: standard basis 1} to rewrite 
	$x^{\nu^{L}}m_{\nu^{R}}[\bar{X}_{[k+1,k+N]}]^{\pm}$
	as a linear combination of 
	$x^{\nu^{L,\circ}}m_{\eta}[\bar{X}_{[l(\nu^{L,\circ})+1,k+N]}^{\pm}]$. 
	Note that this expansion is independent of the 
	rank $k+N$ with fixed coefficients. Then we can apply 
	the result in Theorem \ref{thm: stable limit Koornwinder} to obtain the 
	convergence of the coefficients. The top coefficients 
	statement is obvious.
\end{proof}

\subsection{Stable limit Koornwinder polynomials with tail symmetrization}
Just like the case for type $A$ stable limit Macdonald polynomials, 
the stable limit Koornwinder polynomials $\E_{\lambda}^+$ (and $\E_{\lambda}^-$) 
are respectively linearly independent but not enough to form a basis of $\Pas^{\pm}$.  
To obtain a set of eigenbasis we need to further apply 
the tail-symmetrization operators which are defined as the idempotent operators
$$e_k^{(n)}(t_n) = 
\frac{1}{[n-k]_{t^{-1}}!\prod_{i=1}^{n-k}(1+t_n^{-1} t^{1-i})}
\sum_{w\in S_{1^k,n-k}\rtimes\mathbb{Z}^{n-k}}\chi(T_w)^{-1}T_w$$
for $k\geq 0$ and all $n\geq k$. 

\begin{prop}\label{prop: tail symmetrization}
	\begin{enumerate}
		\item (See \cite{Mac}*{(5.5.7), (5.7.8)})		
		Let $\lambda\in \Pi$. Then 
		$$e_0^{(n)}(t_n)E_{\lambda} 
		= v_{\lambda}(t_n)
		\Pt_{\lambda}(x_{1},...,x_n;q^{-1},t,u_0,t_0,u_n,t_{n})$$
		where
		$$v_{\lambda}(t_n) =
		(\prod_{i\geq 0}\prod_{j=1}^{m_i(\lambda^+)}
		\frac{1-t^{j}}{1-t}
		\prod_{j=1}^{m_0(\lambda^+)}(1+t_{n}t^{j-1}))
		/(\prod_{j=1}^{n}
		\frac{1-t^{j}}{1-t}(1+t_{n}t^{j-1})),$$		
		and $\Pt_{\lambda^+}(X;q^{-1},t,u_0,t_0,u_n,t_{n})$ 
		is the symmetric Koornwinder polynomial associated to the partition 
		$\lambda^+$.
		\item (See \cite{MA}*{Theorem 4.1.2}) 
		Let $\lambda\in \Pi$. Then 
		$$e_0^{(n)}(t_n)x^{\lambda} = v_{\lambda}(t_n)
		\Pt_{\lambda}(x_{1},...,x_n;t,u_n,t_{n})$$
		where 
		$$\Pt_{\lambda}(X;t,u_n,t_{n})
		=v_{\lambda}(t_n)^{-1}\sum_{w\in W_{n}}w(x^{\lambda}\prod_{\alpha\in\Phi_n^+}
		\frac{1-t_{\alpha/2}^{1/2}t_{\alpha}X^{-\alpha}}{1-t_{\alpha/2}^{1/2}X^{-\alpha}})$$
		is the Hall-Littlewood polynomial of type $BC_n$ associated to the partition 
		$\lambda$. By convention we set 
		$t_{\alpha}=1$ if $\alpha\notin \Phi_n^+$.
	\end{enumerate}
\end{prop}

\begin{prop}\label{cor: Hall-Littlewood expansion}
	Let $\lambda\in\mathbb{Z}^n$. Denote 
	$$\mathcal{R}_{\lambda} = \sum_{w\in W_{n}}w(x^{\lambda}\prod_{\alpha\in\Phi_n^+}
		\frac{1-t_{\alpha/2}^{1/2}t_{\alpha}X^{-\alpha}}{1-t_{\alpha/2}^{1/2}X^{-\alpha}})$$
	and suppose for some simple coroot $\alpha_i^{\vee}$ with $0<i\leq n$, 
	$\langle \lambda, \alpha_i\rangle=-m<0$. 
	Then we have
	\begin{enumerate}
		\item If $0<i<n$, then
		$$\mathcal{R}_{\lambda}=t\mathcal{R}_{s_i\lambda}+
		(t-1)\sum_{k=1}^{m-1}\mathcal{R}_{\lambda+k\alpha_i}$$
		\item If $i=n$, then 
		$$\mathcal{R}_{\lambda}=t_n\mathcal{R}_{s_n\lambda}+
		(t_n-1)\sum_{k=1}^{m-1}\mathcal{R}_{\lambda+2k\alpha_n}
		-\beta \sum_{k=1}^{m}\mathcal{R}_{\lambda+(2k-1)\alpha_n}.$$
	\end{enumerate}
\end{prop}
\begin{proof}
	We have $s_i \lambda =\lambda + m\alpha_i$. First we assume $i<n$ and denote 
	$$\Delta'= \prod_{\alpha\in\Phi_n^+\backslash\{\alpha_i\}}
		\frac{1-t_{\alpha/2}^{1/2}t_{\alpha}X^{-\alpha}}{1-t_{\alpha/2}^{1/2}X^{-\alpha}}$$
	Then $s_i$ permutes $\Phi_n^+\backslash\{\alpha_i\}$ and therefore 
	$s_i\Delta' = \Delta'.$	Then we have 
	\begin{align*}
		\mathcal{R}_{\lambda}=&\sum_{w\in W_{n}/\langle s_i\rangle}
		w(\Delta'(x^{\lambda}\frac{1-t X^{-\alpha_i}}{1-X^{-\alpha_i}}
		+x^{s_i\lambda}\frac{1-t X^{\alpha_i}}{1-X^{\alpha_i}}))\\
		=&\sum_{w\in W_{n}/\langle s_i\rangle}
		w(x^{\lambda}\Delta'(\frac{1-t X^{-\alpha_i}}{1-X^{-\alpha_i}}
		+x^{m\alpha_i}\frac{1-t X^{\alpha_i}}{1-X^{\alpha_i}}))\\
		=&\sum_{w\in W_{n}/\langle s_i\rangle}
		w(x^{\lambda}\Delta'(\frac{x_{i+1}^{-m}}{x_i-x_{i+1}}
		((x_i-t x_{i+1})x_{i+1}^m-(x_{i+1}-t x_i)x_i^m)))\\	
		=&\sum_{w\in W_{n}/\langle s_i\rangle}
		w(x^{\lambda}\Delta'(\frac{x_{i+1}^{-m}}{x_i-x_{i+1}}
		(t((x_i-t x_{i+1})x_{i}^m-(x_{i+1}-t x_i)x_{i+1}^m)\\
		&+(t-1)\sum_{i=1}^{\lfloor\frac{m-1}{2}\rfloor}((x_{i+1}x_i)^{m-k}
		((x_i-t x_{i+1})x_{i+1}^k-(x_{i+1}-t x_i)x_{i}^k)\\
		&+((x_i-t x_{i+1})x_{i}^k-(x_{i+1}-t x_i)x_{i+1}^k)))))\\	
		=& t \mathcal{R}_{s_i\lambda} + (t-1)\sum_{k=1}^{m-1}
		\mathcal{R}_{\lambda+k\alpha_i}.
	\end{align*}
	Next we assume $i=n$. Similarly we denote 
	$$\Delta'= \prod_{\alpha\in\Phi_n^+\backslash\{\alpha_n,2\alpha_n\}}
	\frac{1-t_{\alpha/2}^{1/2}t_{\alpha}X^{-\alpha}}{1-t_{\alpha/2}^{1/2}X^{-\alpha}}.$$
	Note that the shifted parameters are $t_{\alpha_n}=a^2$ and 
	$t_{2\alpha_n}=u_n^{-1}$. 
	Then we have 
	{\allowdisplaybreaks
	\begin{align*}
		\mathcal{R}_{\lambda}
		=&\sum_{w\in W_{n}/\langle s_n\rangle}
		w(x^{\lambda}\Delta'(\frac{(1+a x_n^{-1})(1-a^{-1}t_n x_n^{-1})}
		{1-x_n^{-2}}
		+x_n^{2m}\frac{(1+a x_n)(1-a^{-1}t_n x_n)}
		{1-x_n^{2}}))\\
		=&\sum_{w\in W_{n}/\langle s_n\rangle}
		w(\frac{x^{\lambda}\Delta'}{x_n^2-1}
		((x_n^2+\beta x_n-t_n)
		-x_n^{2m}(1+\beta x_n -t_n x_n^2)))\\
		=&\sum_{w\in W_{n}/\langle s_n\rangle}
		w(\frac{x^{\lambda}\Delta'}{x_n^2-1}
		(t_n(x_n^{2m}(x_n^2+\beta x_n-t_n)
		-(1+\beta x_n -t_n x_n^2))\\
		&-\beta\sum_{k=0}^{\lfloor \frac{m-1}{2}\rfloor}
		(x_n^{2k+1}((x_n^{2m-4k-2}(x_n^2+\beta x_n-t_n)
		-(1+\beta x_n -t_n x_n^2))\\
		&+((x_n^2+\beta x_n-t_n)
		-x_n^{2m-4k-2}(1+\beta x_n -t_n x_n^2))
		))\\
		&+(t_n-1)\sum_{k=1}^{\lfloor \frac{m-1}{2}\rfloor}
		(x_n^{2k}((x_n^{2m-4k}(x_n^2+\beta x_n-t_n)
		-(1+\beta x_n -t_n x_n^2))\\
		&+((x_n^2+\beta x_n-t_n)
		-x_n^{2m-4k}(1+\beta x_n -t_n x_n^2))))\\
		&+ \epsilon(m)x_n^m ((x_n^2+\beta x_n-t_n)
		-(1+\beta x_n -t_n x_n^2))))\\
		=& t_n \mathcal{R}_{s_n\lambda} + (t_n-1)\sum_{k=1}^{m-1}
		\mathcal{R}_{\lambda+2k\alpha_n}- \beta \sum_{k=1}^{m}
		\mathcal{R}_{\lambda+(2k-1)\alpha_n}
	\end{align*}}
	where $\epsilon(m)=\beta$ if $m$ is odd and 
	$\epsilon(m)=t_n-1$ if $m$ is even. This is exactly the desired formula.
\end{proof}
\begin{rmk}
	Proposition \ref{cor: Hall-Littlewood expansion} 
	provides a computation formula for 
	the Hall-Littlewood polynomials of type $BC_n$ 
	associated to a general weight.
\end{rmk}

\begin{prop}
For each $k\geq 0$, the sequence of operators 
$\{e_k^{(n)}(t_{\infty})\}_{n\geq k}$ with the fixed parameter $t_{\infty}$ 
is convergent $t$-adically 
to an idempotent operator $e_k^+$ on $\Pas^{\pm}$.
\end{prop}
\begin{proof}
	From Proposition \ref{prop: tail symmetrization} we have 
	for each $\lambda\in \mathbb{Z}^k$ and $\mu\in \Pi$, 
	$$e_k^{(n)}(t_{\infty})x^{\lambda\mu}
	= v_{\mu0^{n-k-s}}(t_{\infty})
	x^{\lambda}\Pt_{\mu}(x_{k+1},...,x_n;t,u_n,t_{n}).$$ 
	Note that
	$$\lim_{n\rightarrow\infty}v_{\mu}(t_{\infty})=
	\prod_{i\geq 0}[m_i(\mu)]_t!(1-t)^{l(\mu)}.$$
	We see that $v_{\lambda_{R}^+}(t_{\infty})$ converges $t$-adically. 
	On the other hand, 
	$$\Pt_{\mu}(x_{k+1},...,x_n;t,u_n,t_{n})
	\rightarrow \Pt_{\mu}(X_k;t,u_n,t_{n})=\Pt_{\mu}(X_k;0,t,0,0,u_n,t_{n})$$
	as $n\rightarrow \infty$. Therefore the sequence
	$$e_k^{(n)}(t_{\infty})x^{\lambda}m_{\mu}[\bar{X}_n^{\pm}]=
	m_{\mu}[\bar{X}_n^{\pm}]e_k^{(n)}(t_{\infty})x^{\lambda}$$
	is convergent for all $\lambda\in\mathbb{Z}^k$ and $\mu\in\Pi$. Furthermore, 
	by Proposition \ref{cor: Hall-Littlewood expansion} we obtain 
	the convergence of 
	$\{e_k^{(n)}(t_{\infty})x^{\lambda\mu}\}_{n\geq k}$ 
	for all $\lambda\in\mathbb{Z}^k$ and $\mu\in \mathbb{Z}^s_{\geq 0}$. Now 
	we consider the case when $\mu\in \mathbb{Z}^s$ and $\mu_s<0$. Then 
	by Proposition \ref{cor: Hall-Littlewood expansion} we have for sufficiently 
	large $n$, 
	$$e_k^{(n)}(t_{\infty})x^{\lambda\mu}
	= O(t^{n-k-s-|\mu_s|})$$
	which converges to 0 as $n\rightarrow \infty$. Thus by noting 
	that $e_k^{(n)}T_i=e_k^{(n)}$ for all $i\geq k$, 
	we obtain the convergence of all sequences of the form 
	$\{e_k^{(n)}(t_{\infty})x^{\lambda\mu}m_{\nu}[\bar{X}_n^{\pm}]\}_{n\geq k}$ 
	for all 
	$\lambda\in\mathbb{Z}^k$, $\mu\in \mathbb{Z}^s$ and $\nu\in \Pi$. 
	By linearity we obtain that the sequence of operators 
	$\{e_k^{(n)}(t_{\infty})\}_{n\geq k}$ is convergent $t$-adically 
	to an idempotent operator $e_k^+$ on $\Pas^{\pm}$.
\end{proof}

\begin{prop}\label{prop: eigenbasis 1}
	For $\lambda\in\mathbb{Z}^k$ with $\lambda_k\neq 0$ and $\mu\in\Pi$, define
	$$\E_{\lambda\vert\mu}^+=
	\frac{1}{(1-t)^{l(\mu)}\prod_{i\geq 0}[m_i(\mu)]_t!}
	\lim_{n\rightarrow\infty}
	e_k^{(n)}(t_{\infty})E_{\lambda\mu 0^{n-k-l(\mu)}}$$
	with respect to the $t$-adic topology. 
	Then $\E_{\lambda\vert\mu}^+$ is an eigenfunction of 
	the positive stable limit Cherednik operators $\tilde{Y}_i^+$ 
	with eigenvalues $\delta_i(\lambda)q^{\lambda_i}
	t^{u_{\lambda\mu}(i)}$. Furthermore, we have
	$$\E_{\lambda\vert\mu}^+=m_{\lambda\vert\mu}+
	\sum_{\nu|\eta<\lambda|\mu}
	\tilde{a}_{\lambda|\mu,\nu|\eta}^+ m_{\nu\vert\eta}$$
\end{prop}
\begin{proof}
	By Proposition \ref{prop: continuity}, for $i\leq k$ we have 
	\begin{align*}
	\tilde{Y}_i^+\E_{\lambda\vert\mu}^+=&
	\lim_{n\rightarrow\infty}t^{n}Y_i^{(n)}
	e_k^{(n)}(t_{\infty})E_{\lambda\mu 0^{n-k-l(\mu)}}\\
	=&\frac{1}{(1-t)^{l(\mu)}\prod_{i\geq 0}[m_i(\mu)]_t!}\lim_{n\rightarrow\infty}e_k^{(n)}(t_{\infty})
	t^{n}Y_i^{(n)}E_{\lambda\mu 0^{n-k-l(\mu)}}\\
	=&\frac{1}{(1-t)^{l(\mu)}\prod_{i\geq 0}[m_i(\mu)]_t!}\lim_{n\rightarrow\infty}e_k^{(n)}(t_{\infty})
	\delta_i(\lambda)q^{\lambda_i}
	t^{u_{\lambda\mu}(i)}E_{\lambda\mu 0^{n-k-l(\mu)}}\\
	=&\delta_i(\lambda)q^{\lambda_i}
	t^{u_{\lambda\mu}(i)}\E_{\lambda\vert\mu}^+
	\end{align*}
	For $i>k$, since $\E_{\lambda\vert\mu}\in\P(k)^{\pm}$ we have 
	$\tilde{Y}_i^+\E_{\lambda\vert\mu}=0.$ 
	
	Now consider the monomial expansion
	$$E_{\lambda\mu 0^{n-k-l(\mu)}}
	=X^{\lambda\mu 0^{n-k-l(\mu)}}+
	\sum_{\nu<\lambda\mu}a_{\lambda\mu 0^{n-k-l(\mu)},\nu}^{(n)}
	X^{\nu}.$$
	Note that $\lambda\mu 0^{n-k-l(\mu)}$ is the 
	unique minimal element in the orbit of $\lambda\mu 0^{n-k-l(\mu)}$ 
	under the action of $S_{1^k,n-k}\rtimes\mathbb{Z}^{n-k}$. 
	Futher note that the Demazure-Lusztig operator $T_i$ 
	for $i\geq k$ acting on monomials only generates monomials in the same 
	$S_{1^k,n-k}\rtimes\mathbb{Z}^{n-k}$-orbit 
	or monomials in a strict lower orbit. Thus by Proposition 
	\ref{prop: Bruhat order} we have 
	$$\nu^{L}\vert(\nu^{R})^+<\lambda\vert\mu$$
	where $\nu^{L}=(\nu_1,...,\nu_{k'})$ with 
	$k'\leq k$ and 
	$$\nu_{k'}\neq 0,\quad \nu_{k'+1}=...=\nu_k=0,$$
	and $\nu^{R}=(\nu_{k+1},...,\nu_{n})$. 
	Furthermore, by Proposition \ref{prop: tail symmetrization} 
	the coefficient of $X^{\lambda\mu 0^{n-k-l(\mu)}}$ in 
	$e_k^{(n)}(t_{\infty})E_{\lambda\mu 0^{n-k-l(\mu)}}$ can only 
	come from the contribution of $e_k^{(n)}(t_{\infty})X^{\lambda\mu 0^{n-k-l(\mu)}}$
	and therefore is $v_{\mu}(t_{\infty})$. 
	Now that we have $t$-adic convergence
	$$\lim_{n\rightarrow\infty}v_{\mu}(t_{\infty})=
	\prod_{i\geq 0}[m_i(\mu)]_t!(1-t)^{l(\mu)},$$
	the normalizer in the definition makes 
	the coefficient of $m_{\lambda\vert\mu}$ equal 1. 
	This shows the triangularity and the monicity of the expansion of 
	$\E_{\lambda\vert\mu}^+$ in the basis $\{m_{\nu\vert\eta}\}$.
\end{proof}

For negative case, we first note the following proposition which is 
a direct consequence of 
Proposition \ref{prop: recursion of Koornwinder}.
\begin{prop}\label{cor: eigenbasis 2}
	Let $\lambda\in\mathbb{Z}^k$ and $\mu\in\mathbb{Z}^{n-k}$. 
	Then we have
	$$e_k^{(n)}(t_{\infty})E_{\lambda\mu}=
	\gamma_{\lambda\mu}^{(n)}e_k^{(n)}(t_{\infty})E_{\lambda\mu^+ }$$
	for some $\gamma_{\lambda\mu}^{(n)}\in \mathbb{K}'$. 
\end{prop}
\begin{proof}
	Note that $e_k^{(n)}(t_{\infty})T_i = e_k^{(n)}(t_{\infty})$ for all $i\geq k$. 
	Then we can apply the symmetrizer to the formulae in 
	Proposition \ref{prop: recursion of Koornwinder} to obtain the result 
	by induction on the Bruhat order of $\mu$. 
\end{proof}
\begin{ex}\label{ex: normalizer for negative}
	For the special case $\tau = (\lambda,-\mu,0^{n})$ where 
	$\lambda\in\mathbb{Z}^k$ and $\mu\in\Pi$ with $l(\mu)=s$ and 
	the first appearance of each 
	distinct part denoted by $\mu_{a_1}>...>\mu_{a_r}$, we
	consider the longest element (See 
	\cite{Hum}*{Section 3.19})
	$$w_0=(s_{k+1}...s_n)^{n-k}\in S_{1^k,n-k}\rtimes\mathbb{Z}^{n-k}$$
	Then 
	$$w_0(\lambda,-\mu,0^{n}) = (\lambda,\mu,0^{n}).$$
	By recursively using the formula in Proposition 
	\ref{prop: recursion of Koornwinder} we obtain that
	{\allowdisplaybreaks
	\begin{align*}
		\gamma_{\tau}^{(n)}=&
		\prod_{j=0}^{n-1}\prod_{i=1}^{s}
		(\frac{1-q^{-\mu_i}t_0 t_{\infty}t^{\beta_{\lambda\mu}(k+i)+2n-j}}
		{1-q^{-\mu_i}t_0 t_{\infty}t^{\beta_{\lambda\mu}(k+i)+2n-j-1}}
		\frac{1-q^{-\mu_i}t^{\beta_{\lambda\mu}(k+i)+n-j+1}}
		{1-q^{-\mu_i}t^{\beta_{\lambda\mu}(k+i)+n-j}})\\
		&\prod_{i=1}^{s}\frac{(1-q^{-\mu_i}t_0 t_{\infty}
		t^{\beta_{\lambda\mu}(k+i)+n})(1-q^{-\mu_i}t_{\infty}
		t^{\beta_{\lambda\mu}(k+i)+n})}
		{1-q^{-2\mu_i}t_0 t_{\infty}t^{2\beta_{\lambda\mu}(k+i)+2n}}\\
		&\prod_{i<j\leq s}\frac{1-q^{-\mu_i-\mu_j}t_0 t_{\infty}
		t^{\beta_{\lambda\mu}(k+i)+\beta_{\lambda\mu}(k+j)+2n+1}}
		{1-q^{-\mu_i-\mu_j}t_0 t_{\infty}
		t^{\beta_{\lambda\mu}(k+i)+\beta_{\lambda\mu}(k+j)+2n}}\\
		&\prod_{i<j\leq r}\left(\frac{1-q^{-\mu_{a_i}+\mu_{a_j}}
		t^{\beta_{\lambda\mu}(k+a_i)-\beta_{\lambda\mu}(k+a_j)+m_{a_j}(\mu)}}
		{1-q^{-\mu_{a_i}+\mu_{a_j}}
		t^{\beta_{\lambda\mu}(k+a_i)-\beta_{\lambda\mu}(k+a_j)}}\right)^{m_{a_i}(\mu)}\\
		=&
		\prod_{i=1}^{s}\frac{(1-q^{-\mu_i}t_0 t_{\infty}
		t^{\beta_{\lambda\mu}(k+i)+2n})(1-q^{-\mu_i}t_{\infty}
		t^{\beta_{\lambda\mu}(k+i)+n})
		(1-q^{-\mu_i}t^{\beta_{\lambda\mu}(k+i)+n+1})}
		{(1-q^{-2\mu_i}t_0 t_{\infty}t^{2\beta_{\lambda\mu}(k+i)+2n})
		(1-q^{-\mu_i}t^{\beta_{\lambda\mu}(k+i)+1})}\\
		&\prod_{i<j\leq s}\frac{1-q^{-\mu_i-\mu_j}t_0 t_{\infty}
		t^{\beta_{\lambda\mu}(k+i)+\beta_{\lambda\mu}(k+j)+2n+1}}
		{1-q^{-\mu_i-\mu_j}t_0 t_{\infty}
		t^{\beta_{\lambda\mu}(k+i)+\beta_{\lambda\mu}(k+j)+2n}}\\
		&\prod_{i<j\leq r}\left(\frac{1-q^{-\mu_{a_i}+\mu_{a_j}}
		t^{\beta_{\lambda\mu}(k+a_i)-\beta_{\lambda\mu}(k+a_j)+m_{a_j}(\mu)}}
		{1-q^{-\mu_{a_i}+\mu_{a_j}}
		t^{\beta_{\lambda\mu}(k+a_i)-\beta_{\lambda\mu}(k+a_j)}}\right)^{m_{a_i}(\mu)}
	\end{align*}}
	Note that the asymptotic degree in $t$ of $\gamma_{\tau}^{(n)}$ is 
	\begin{align*}
		& 2sn+\frac{s(s-1)}{2}+\sum_{i<j}m_{a_i}(\mu)m_{a_j}(\mu)\\
		=&
		\sum_{i=1}^s (n+s-i)+\rm{deg}_t 
		\begin{pmatrix}
			n+s\\
			m_0(\mu 0^{n});m_{a_1}(\mu 0^{n});...;m_{a_r}(\mu 0^{n})
		\end{pmatrix}_t
	\end{align*}
	This is exactly the degree of $v_{\mu 0^{n}}(t_{\infty})^{-1}$ in $t$. 
	By a direct computation of the $t^{-1}$-adic limit we may obtain the following result.
\end{ex}

\begin{cor}\label{cor: eigenbasis minus}
	For $\lambda\in\mathbb{Z}^k$ with $\lambda_k\neq 0$ and $\mu\in\Pi$ 
	with $l(\mu)=s$ and the first appearance of each 
	distinct part denoted by 
	$\mu_{a_1}>...>\mu_{a_r}>0$ 
	define
	\begin{align*}
		\E_{\lambda\vert\mu}^-&=
		\frac{\prod_{i=1}^{r}(1-q^{\mu_{a_i}}
		t^{-\beta_{\lambda\mu}(k+a_i)-1})}
		{(1-t^{-1})^{s}\prod_{i>0}[m_i(\mu)]_{t^{-1}}!}
		\\
		&\prod_{i<j\leq r}
		\left(\frac{1-q^{-\mu_{a_i}+\mu_{a_j}}
			t^{\beta_{\lambda\mu}(k+a_i)-\beta_{\lambda\mu}(k+a_j)}}
			{1-q^{-\mu_{a_i}+\mu_{a_j}}
			t^{\beta_{\lambda\mu}(k+a_i)-\beta_{\lambda\mu}(k+a_j)+m_{a_j}(\mu)}}\right)^{m_{a_i}(\mu)}
		\lim_{n\rightarrow\infty}
		e_k^{(n)}(t_{\infty})E_{\lambda,-\mu, 0^{n-k-s}}		
	\end{align*}
	with respect to the $t^{-1}$-adic topology. 
	Then $\E_{\lambda\vert\mu}^-$ is an eigenfunction of 
	the negative stable limit Cherednik operators $\tilde{Y}_i^-$ 
	with eigenvalues $\delta_i'(\lambda)(t_0t_n)q^{\lambda_i}
	t^{-u_{\lambda\mu}(i)}$. Furthermore, the monomial expansion of 
	$\E_{\lambda\vert\mu}^-$ is 
	$$\E_{\lambda\vert\mu}^-=
	m_{\lambda\vert\mu}+ \sum_{\nu|\eta<\lambda|\mu} 
	\tilde{a}_{\lambda|\mu,\nu|\eta}^- m_{\nu\vert\eta}$$
\end{cor}

As a direct consequence of Proposition \ref{prop: eigenbasis 1} 
we obtain the following theorem which is the main result of this section.
\begin{thm}\label{thm: positive stable limit Koornwinder}
	The set of positive stable limit Koornwinder polynomials 
	$\{\E_{\lambda\vert\mu}^{+}\}$ 
	forms a simultaneous eigenbasis for $\Pas^{\pm}$ of 
	the positive stable limit Cherednik operators $\tilde{Y}_i^+$ 
	with eigenvalues $\delta_i(\lambda)q^{\lambda_i}t^{u_{\lambda\mu}(i)}$.
\end{thm}
\begin{proof}
	Note that 
	the transition matrix $\{\tilde{a}_{\lambda\mu,\nu\eta}^{+}\}$ 
	from $\{\E_{\lambda\vert\mu}^{+}\}$ to $\{m_{\nu\vert\eta}\}$ is 
	lower triangular with monic diagonal entries. 
	Therefore by Proposition \ref{prop: standard basis 1} 
	the set $\{\E_{\lambda\vert\mu}^{+}\}$ forms a basis of $\Pas^{\pm}$. 
	By Proposition \ref{prop: eigenbasis 1} we see that each 
	$\E_{\lambda\vert\mu}^{+}$ is a simultaneous eigenfunction of 
	the positive stable limit Cherednik operators $\tilde{Y}_i^+$ 
	with eigenvalues $\delta_i(\lambda)q^{\lambda_i}t^{u_{\lambda\mu}(i)}$.
\end{proof}

\begin{ex}
	We have 
\begin{align*}
	E_{(1,1,0^n)}=& x_1x_2+\frac{\alpha+t^n t_0\beta}{q-t^{2n}t_0t_n}(x_1+x_2)\\
	&+\frac{t_0(1-t)}{q-t^{2n+1}t_0t_n}
	+\frac{(\alpha+t^n t_0\beta)(\alpha+t^{n+1} t_0\beta)}
	{(q-t^{2n}t_0t_n)(q-t^{2n+1}t_0t_n)}
\end{align*}
and the partial symmetrization
\begin{align*}
	e_{1}^{(n+2)}(t_{\infty})E_{(1,1,0^n)}=&
	\frac{1}{[n+1]_t(1+t^n t_{\infty})}(x_1 e_1[X_1^{\pm}]
	+\frac{\alpha+t^n t_0\beta}{q-t^{2n}t_0t_{\infty}}e_1[X_1^{\pm}])\\
	&+\frac{(1+t^n t_{\infty})\alpha+(q+t^n t_0)\beta}
	{(1+t^n t_{\infty})(q-t^{2n}t_0t_{\infty})}x_1\\
	&+\frac{\alpha^2+(1-t)qt_0+t^nt_0\alpha\beta+t^{2n}t_0^2 t_{\infty}(t-1)}
	{(q-t^{2n}t_0t_{\infty})(q-t^{2n+1}t_0t_{\infty})}\\
	&+\frac{q\alpha\beta+qt_0t^n\beta+t^{2n}\alpha\beta+t^{n+1}t_0\alpha\beta
	+t^{2n+1}t_0^2\beta^2}
	{(1+t^n t_{\infty})(q-t^{2n}t_0t_{\infty})(q-t^{2n+1}t_0t_{\infty})}	
\end{align*}
The positive stable limit Koornwinder polynomial is then
\begin{align*}
	\E_{(1\vert 1)}^+=&x_1 e_1[X_1^{\pm}]+q^{-1}\alpha e_1[X_1^{\pm}]
	+\frac{q^{-1}\alpha+\beta}{1-t}x_1+
	\frac{q^{-2}\alpha^2+q^{-1}(\alpha\beta+t_0-tt_0)}{1-t}
\end{align*}
\end{ex}

By the same argument we can also obtain the following result for the 
negative stable limit Koornwinder polynomials.
\begin{thm}\label{thm: negative stable limit Koornwinder}
	The set of negative stable limit Koornwinder polynomials 
	$\{\E_{\lambda\vert\mu}^{-}\}$ 
	forms a simultaneous eigenbasis for $\Pas^{\pm}$ of 
	the negative stable limit Cherednik operators $\tilde{Y}_i^-$ 
	with eigenvalues 
	$\delta_i'(\lambda)(t_0t_n)q^{\lambda_i}t^{-u_{\lambda\mu}(i)}$.
\end{thm}

\begin{rmk}
	In fact, from Proposition \ref{prop: eigenbasis 1}, 
	Corollary \ref{cor: eigenbasis minus},  
	Theorem \ref{thm: positive stable limit Koornwinder} and 
	Theorem \ref{thm: negative stable limit Koornwinder} 
	we obtain that 
	the positive and negative stable limit Koornwinder polynomials 
	$\E_{\lambda\vert\mu}^{\pm}$ are uniquely 
	characterized by the following properties:
	\begin{enumerate}
		\item The transition matrix into the monomial basis 
		$\{m_{\nu\vert\eta}\}$ 
		is lower triangular with monic diagonal entries, namely 
		$$\E_{\lambda\vert\mu}^{\pm}\in m_{\lambda\vert\mu}+
		\sum_{\nu\vert\eta<\lambda\vert\mu} \mathbb{K}' m_{\nu\vert\eta};$$
		\item $\E_{\lambda\vert\mu}^{\pm}$ is a simultaneous eigenfunction of 
		the positive (resp. negative) stable limit Cherednik operators 
		$\tilde{Y}_i^{\pm}$.
	\end{enumerate}
\end{rmk}

\begin{ex}
	We have 
\begin{align*}
	&E_{(1,-1,0^n)}\\
	=& x_1x_2^{-1}+
	\frac{q(t-1)}{t^2-q}x_1 e_1[\bar{X}_{[3,n+2]}^{\pm}]+
	\frac{\alpha+t^n t_0\beta}{q-t^{2n}t_0 t_{\infty}}x_2^{-1}
	-\frac{q(t-1)(\alpha+t^n t_0\beta)}{(t^2-q)(t^{2n}t_0t_{\infty}-q)}
	e_1[\bar{X}_{[3,n+2]}^{\pm}]\\
	&+\frac{q(t^{2n+3}t_{\infty}(t_0-1)+q^2(1-t-t^n t_{\infty}+t^{n+1})
	+qt^{n+1}(-t+t_{\infty}+t^{n+1}t_{\infty}-t^n t_0t_{\infty}))
	}{(q-t^2)(q^2-t^{2n+2}t_0t_{\infty})}\\
	&(x_1x_2
	+\frac{\alpha+t^n t_0 \beta}{q-t^{2n}t_0t_{\infty}}x_2)
	+\frac{q-t^{n+2}}{(q-t^2)(q-t^{2n}t_0t_{\infty})
	(q^2-t^{2n+2}t_0t_{\infty})}\\
	&((-t^{3n+2}t_0t_{\infty}^2\alpha+q^3\beta-qt^{n+1}t_{\infty}
	((-1+t^n t_0)\alpha+t^n t_0(-1+t+t^n t_0)\beta\\
	&+q^2(\alpha+t^n t_0(1-t^n t_{\infty})\beta)))x_1
	+(t^{n+1}t_{\infty}(t^{2n}(t-1)t_0^2 t_{\infty}
	+\alpha^2+t^n t_0 \alpha\beta)\\
	&+q(\alpha^2+t^n t_0((1+t)\alpha\beta+t(1-t)t_{\infty})
	+t^{2n}t_0((t-1)t_0t_{\infty}+tt_0\beta^2))))
\end{align*}
and the negative stable limit of the normalized partial symmetrization operator 
$e_1^{(n)}$ is 
\begin{align*}
	\E_{(1\vert 1)}^-=&
	x_1 e_1[\bar{X}_1^{\pm}]-\frac{t_{0}^{-1}\alpha
	+t_{\infty}^{-1}\beta}{1-t^{-1}} x_1+1
\end{align*}
\end{ex}

\subsection{Parameter specializations}
From the Koornwinder polynomials of type $(C_n^{\vee},C_n)$ 
we can obtain other types of the Koornwinder polynomials 
by appropriate specializations of parameters. 
These specializations are compatible with the stable limit. Hence 
we list the parameter specializations of stable limit Koornwinder 1polynomials 
of other affine root systems in Table \ref{tab:parameter-specialization} for 
reference.  
We refer to 
\cite{Mac}*{Section 1.3}, 
\cite{IS}*{Section 11}, \cite{YY} for detailed discussions of the specialization.
\begin{table}[H]
\begin{center}
	\begin{tabular}{|c|c|c|}
	\hline
	Type                & Generating parameters                  & Other parameters         \\ \hline
	$(C_n^{\vee},C_n)$  & $(q,t,t_0,t_n,a,c)$                    & $(\alpha,\beta)$         \\ \hline
	$(C_n^{\vee},BC_n)$ & $(q,t,t_0,t_n,a,q^{1/2})$              & $(q^{1/2}(1-t_0),\beta)$ \\ \hline
	$(BC_n,C_n)$        & $(q,t,t_0,t_n,a,(qt_0)^{1/2})$         & $(0,\beta)$              \\ \hline
	$D_{n+1}^{(2)}=C_n^{\vee}$        & $(q,t,t_0,t_n,1,q^{1/2})$              & $(q^{1/2}(1-t_0),1-t_n)$ \\ \hline
	$(B_n^{\vee},B_n)$  & $(q,t,1,t_n,a,q^{1/2})$                & $(0,\beta)$              \\ \hline
	$A_{2n}^{(2)}=BC_n$              & $(q,t,t_0,t_n,1,(qt_0)^{1/2})$         & $(0,1-t_n)$              \\ \hline
	$C_n^{(1)}=C_n$               & $(q,t,t_0,t_n,t_n^{1/2},(qt_0)^{1/2})$ & $(0,0)$                  \\ \hline
	$B_n^{(1)}=B_n$               & $(q,t,1,t_n,1,q^{1/2})$                & $(0,1-t_n)$              \\ \hline
	$A_{2n-1}^{(2)}=B_n^{\vee}$        & $(q,t,1,t_n,t_n^{1/2},q^{1/2})$        & $(0,0)$                  \\ \hline
	$D_n^{(1)}=D_n$               & $(q,t,1,1,1,q^{1/2})$                  & $(0,0)$                  \\ \hline
	\end{tabular}
\end{center}
\caption{Specialization of parameters for stable limit Koornwinder polynomials 
to other affine root systems}
\label{tab:parameter-specialization}
\end{table}

\section*{Acknowledgements}
The author would like to thank Bogdan Ion for 
helpful discussions and suggestions.

\end{document}